\documentclass{ws-ijm}
\usepackage[all]{xy}
\usepackage{enumerate}

\newtheorem{thm}{Theorem}[section]
\newtheorem{prop}[thm]{Proposition}
\newtheorem{coro}[thm]{Corollary}
\newtheorem{lema}[thm]{Lemma}
\newtheorem{rmk}[thm]{Remark}
\newtheorem{defin}[thm]{Definition}
\newtheorem{notat}[thm]{Notation}

\newcommand{\Ad}{\textnormal{Ad}}
\newcommand\id{{\operatorname{id}}}
\newcommand{\diag}{\operatorname{diag}}
\newcommand{\rank}{\operatorname{rank}}
\newcommand{\tr}{\operatorname{tr}}
\newcommand{\Tr}{\operatorname{Tr}}
\newcommand{\lcm}{\operatorname{lcm}}

\begin{document}
\markboth{Francisco Torres-Ayala}
 {Primitivity of amalgamated full free products}

\catchline{}{}{}{}{}

\title{CONDITIONS FOR PRIMITIVITY OF UNITAL AMALGAMATED FULL FREE PRODUCTS OF \\
FINITE DIMENSIONAL C*-ALGEBRAS}

\author{Francisco Torres-Ayala}

\address{Instituto de Matem\'aticas, UNAM Campus Juriquilla, \\ Boulevard Juriquilla 3001, Santiago de Quer\'etaro C.P. 76230,\\
Quer\'etaro, M\'exico \\ \email{tfrancisco@matem.unam.mx} }

\maketitle

\begin{abstract}
We consider amalgamated unital full free
products of the form $A_1*_DA_2$, where $A_1, A_2$ and $D$ are finite dimensional C*-algebras and
there are faithful traces on $A_1$ and $A_2$ whose restrictions to $D$ agree. We provide several
conditions on the matrices of partial multiplicities of the inclusions $D\hookrightarrow A_1$
and $D\hookrightarrow A_2$ that guarantee that the C*-algebra $A_1*_DA_2$ is primitive. If the ranks
of the matrices of partial multiplicities are one, we prove that the algebra $A_1*_DA_2$ is primitive
if and only if it has a trivial center.

\end{abstract}

\keywords{Primitive C*--algebra, Amalgamated Full Free Products}
\ccode{Mathematics Subject Classification 2000: 46L05, 46L059}

\section{Introduction and preliminaries}\label{SectionIntroduction}

A C*-algebra is called primitive if it admits a faithful and irreducible $*$-representation i.e. a $*$-homomorphism
from the algebra to the bounded operators on some Hilbert space such  that it is an isometry and the only closed invariant
subspaces, for its  image, are the trivial ones. If one takes the point of view of using the Jacobson topology to study the structure
of a C*-algebra, then primitive ones are the building blocks. For  discussions and  examples see ~\cite{Kaplansky}, ~\cite{Bedos&Omland-Amenable}, ~\cite{Bedos&Omland-Modular}, ~\cite{Choi} and ~\cite{Murphy}. In ~\cite{Dykema&Torres} we  proved that, under the assumption of
residually finite dimensionality, the only non trivial example of a unital full free product of RFD C*-algebras that failed to be primitive is $\mathbb{C}^2*_{\mathbb{C}}\mathbb{C}^2$. Another way to state this result is to say that the only obstruction, for a unital full free product of RFD C*-algebras, to be primitive is that it has a non trivial center. In this article we give a partial generalization. We now study the amalgamated unital full free product of matrix algebras and provide a criteria that
guarantee the primitivity of $A_1*_DA_2$.  The main results are the following.

\begin{thm}\label{Thm:MainThm}
Consider  unital amalgamated full free products of the form $A_1*_DA_2$ where we assume:
\begin{enumerate}[(i)]
\item there are faithful traces on $A_1$ and $A_2$ whose restrictions to $D$ agree,
\item the ranks of the matrices of partial multiplicities of the inclusions $D \hookrightarrow A_s$, $s=1,2$, are one.
\end{enumerate} 
If the center of $A_1*_DA_2$ is trivial then $A_1*_DA_2$ is primitive.
\end{thm}

To state the second result, we need a condition that we called LP condition (see Definition \ref{Def:LPcondition}).

\begin{thm}\label{CoroPrimitivityBiggerThan2}
Let $A_1,A_2$ and $D$ be finite dimensional C*-algebras. Assume that:
\begin{enumerate}[(i)]
\item there are faithful traces on $A_1$ and $A_2$ whose restrictions to $D$ agree,
\item all the entries of $\mu_s$ (the matrices of partial  multiplicities of the inclusions $D\hookrightarrow A_s$), $s=1,2$,  are either zero or grater or equal than 2,
\item $A_1*_DA_2$ satisfy the LP condition.
\end{enumerate}
 Then $A_1*_DA_2$ is primitive. 
\end{thm}

In fact, there is a more general version (Theorem \ref{Thm:MultiplicitiesBiggerThan2}) but it is a little bit technical
to be  presented at the introduction.

This is the right place for some notes regarding the proof of Theorem \ref{Thm:MainThm}. Firstly,  we give a  criteria that guarantees the primitivity of products of the form $A_1*_DA_2$,  where $A_1*_DA_2$ is RFD and $A_1,A_2$ and $D$ are finite dimensional. Then, for all the cases in which  our criteria does not hold, we show we have a non trivial center.

\subsection{Preliminaries}

Most of the time,  $A_1, A_2$ and $D$ will denote finite
dimensional C*-algebras and by  $\gamma_s:D \to A_i$, $s=1,2$,  we will denote unital inclusions, that is
unital, injective $*$-homomorphisms. Only at the final section, we will specialize to the case when the matrices of partial multiplicities  have rank one.  With respect this inclusions we take, $(A_1*_DA_2,\iota_1,\iota_2)$,  the unital  full free product with amalgamation over $D$, in short denoted by $A_1*_DA_2$. 

For a positive integer $n$, $M_n$ denotes the algebra of $n \times n$ matrices with entries in $\mathbb{C}$.

For $s=0,1,2$ we let $l_s$ denote the dimension of the center of $D$, $A_1$ and $A_2$ respectively. At some later point we will have to perform computations 
using  the dimensions of the direct summands of $A_1$ and $A_2$. Thus, for the rest of the exposition, we fixed an order for the direct summands and, with respect to this order, 
we denote $n_s(i)$, $i=1,\dots, l_s$, $s=1,2$, the dimensions of the direct summand of $A_s$. With this notation we
have that $A_s$ is $*$-isomorphic to $\oplus_{i=1}^{l_s}M_{n_s(i)}$. 

Later on, it will  become clear that the primitivity of $A_1*_DA_2$ only depends on the way how we glue $D$ to $A_1$ and $A_2$. To be more precise, it will depend on the matrix of partial multiplicities of the inclusions $\gamma_s$. Since these matrices will be important they will be denoted as $\mu_s$ and its $(i,j)$-th entry will be denoted as $\mu_s(i,j)$. In general, if $D$ is a unital $*$-subalgebra of a finite dimensional C*-algebra $A$, $\mu(A,D)$ denotes the matrix of partial multiplicities of the inclusion $D \hookrightarrow A$.

Given a $*$-representation $\pi:A_1*_DA_2 \to \mathbb{B}(H)$ we take  $\pi^{(s)}:=\pi \circ \iota_s $, $s=1,2$,
and $\pi^{(0)}=\pi \circ \iota_1 \circ \gamma_1= \pi \circ \iota_2 \circ \gamma_2$. Thus, we  might think $\pi^{(s)}$ and $\pi^{(0)}$ as the restrictions of $\pi$ to $A_s$ and $D$ respectively. 

For a positive integer $n$ we let $[n]$ denote the set $\{1.2,\dots, n-1,n\}$.

The article is divided as follows: section two deals with two
important simplifications, one of them is the criteria that we mentioned above. Section three deals with finite dimensional C*-algebras in general position and finally section four provides a proof of Theorem \ref{Thm:MainThm} and
Theorem \ref{CoroPrimitivityBiggerThan2}.

\section{Two important simplifications}

In this section we present two important simplifications of our main problem i.e. the primitivity of $A_1*_DA_2$. For most of
this section we don't need to assume $A_1$, $A_2$ and $D$ to be finite dimensional. When we need it we state it clearly.

\subsection{Reduction to abelian $D$}

It turns out that it is enough to consider abelian $D$.  This follows from Lemma 2.1  in ~\cite{Brown&Dykema}. Since we will
 use a minor modification we  give a proof of it.

\begin{prop}\label{Prop:Compression}
Let $A$ be a unital C*-algebra and assume there is a projection $p$ and partial isometries $v_1,\dots, v_n$ in
$A$ such that $v_i^*v_i \leq p$ and $\sum_{i=1}^n v_iv_i^*=1-p$. Then:
\begin{enumerate}[(i)]
\item If $pAp$ is primitive so is $A$.
\item If $pAp$ has non trivial center so does $A$. 
\end{enumerate}
\end{prop}
\begin{proof}
For convenience let $v_0=p$. The assumption $v_i^*v_i \leq p$ implies that, for any $a\in A$ and any $i$ and $j$, $v_i^*av_j$ lies in $pAp$. Now define a map $\varphi:A\to M_{n+1}(pAp)$ by $\varphi(a)[i,j]=v_i^*av_j$. Thus $\varphi$ is well
defined, linear, preserves adjoints and the identity $\sum_{i=0}^n v_iv_i^*=p$ implies that $\varphi$ is multiplicative and injective. Even more,
if $Q=\diag( (v_i^*v_i)_{i=0}^n )$ then $\varphi$ is a $*$-isomorphism form $A$ onto $QM_{n}(pAp)Q$. Indeed, this follows
from the fact that $\varphi(v_{i_0}av_{j_0}^*)$ is the matrix with all entries equal to zero except the $(i_0,j_0)$-th entry which
equals $v_{i_0}^*v_{i_0}av_{j_0}^*v_{j_0}$. 

Now we prove (i). It is known that hereditary C*-subalgebras of primitive ones are primitive as well (exercise 4, section 3, chapter III, ~\cite{Takesaki}). Thus we only need to show $M_{n+1}(pAp)$ is primitive. But this is true if we assume $pAp$ to be primitive. 

To prove (ii) we need a little bit more of work. To start, take $x$ a non trivial element in the center of $pAp$. Now
let $y=\sum_{i=0}^n \varphi(v_ixv_i^*)$. Notice that $y$ is a diagonal element in $M_{n+1}(pAp)$. Since $v_0xv_0^*=pxp=x$,
$y$ is not trivial. Next we show that, for any $a$ in $A$, the $(i,j)$-th entry of $\varphi(a)y$ and $y\varphi(a)$ are the same. A
direct computation shows
\begin{eqnarray*}
(\varphi(a)y)[i,j]&=&\sum_{s=0}^n\varphi(a)[i,s]y[s,j]=v_i^*av_jxv_j^*v_j, \\
(y\varphi(a))[i,j]&=&\sum_{s=0}^n y[i,s]\varphi(a)[s,j]=v_i^*v_ixv_i^*av_j .
\end{eqnarray*}

Using that $v_i^*av_j$ and $v_i^*v_i$ lie in $pAp$ 
and $x$ in its center, it follows that
\begin{eqnarray*}
v_i^*av_jxv_j^*v_j&=& xv_i^*av_jv_j^*v_j=xv_iav_j=xv_iv_i^*v_iav_j=  v_i^*v_ixv_i^*av_j
\end{eqnarray*}

and this finishes the proof.
\end{proof}

\begin{coro}\label{Coro:ReductionTrivialCenter}
Let $A_1,A_2$ and $D$ denote C*-algebras and assume $D$ is finite dimensional. For each direct
summand of $D$, choose a minimal projection in that direct summand and let
$p$ denote their sum. Then:
\begin{enumerate}[(i)]
\item If $pA_1p*_{pDp} pA_2p$ is primitive, then $A_1*_DA_2$ is primitive as well.
\item If $pA_1p*_{pDp}pA_2p$ has a non trivial center, then so does $A_1*_DA_2$.
\end{enumerate}
\end{coro}
\begin{proof}

Fix and order for the direct summands of $D$ and for the $k$-th direct summand of $D$ let $\{ e_{i,j}^{(k)}\}_{i,j}$ denote a system of  matrix units. If necessary, we make a change of basis so that $e^{(k)}_{1,1}$ is the minimal projection taken at
the beginning. Then, the partial isometries $\{e_{i,1}^{(k)}\}_{i,k}$ full fill the hypothesis of
Proposition \ref{Prop:Compression}. Lastly, from  Lemma 2.2 in ~\cite{Brown&Dykema} it follows that $p(A_1*_DA_2)p\simeq pA_1p*_{pDp}pA_2p$.

\end{proof}

We need one last lemma to really reduce our problem to the case when $D$ is abelian. 

\begin{lema}\label{Lemma:CompresionSameMatrix}
Assume $A$ and $D$ are finite dimensional and we have an inclusion $\gamma: D \to A$ .
For each direct
summand of $D$, choose a minimal projection in that direct summand and let $p$ denote their sum. Then
\[
\mu(A,D)=\mu(\gamma(p)A\gamma(p),pDp)
\]
\end{lema}

\begin{proof}
As before, let $\{e_{i,j}^{(k)}\}$, be a system of matrix units for $D$. With no loss of generality we may assume 
$p=\sum_{k}e_{1,1}^{(k)}$. Let $\mu(i,j)$ and $\tilde{\mu}(i,j)$ denote the $(i,j)$-th entry of $\mu(A,D)$ and
$\mu(\gamma(p)A\gamma(p),pDp)$ respectively. By definition $\mu(i,j)$ is the rank of $\pi_i((\gamma(e_{1,1}^{(j)}))$,
and $\tilde{\mu}(i,j)$ is the rank of $\gamma(p)\pi_i(\gamma(e_{1,1}^{(j)}))\gamma(p)$, where $\pi_i$ denotes the  projection from $A$ onto the $i$-th direct summand of $A$.  Since
$$
\gamma(p)\pi_i( \gamma(e_{1,1}^{(j)} )) \gamma(p)=\pi_i(\gamma(p))\pi_i(\gamma(e_{1,1}^{(j)}))\pi_i(\gamma(p)) =\pi_i(\gamma(e_{1,1}^{(j)}))
$$
this finishes the proof.
\end{proof}

From the previous proposition, lemma and corollary is clear that, to prove Theorem \ref{Thm:MainThm}, we can restrict ourselves  to the case where $D$ is abelian.

\subsection{Finite dimensional criteria}

The second simplification give us a finite dimensional criteria that guarantee the primitivity of $A_1*_DA_2$, when
$A_1*_DA_2$ is assumed to be residually finite dimensional. From the works in  ~\cite{Brown&Dykema} and  ~\cite{Dykema&Exel&Etal-EmbeddingsFullFreeProd} , it is known that $A_1*_DA_2$ is
 is RFD if and only if there are faithful sates on $A_1$ and $A_2$ that agree on $D$.  Thus, in this subsection we 
assume there are such states. Now that we have restrict to the residually finite dimensional case,  it is not surprise
to direct our efforts to study finite dimensional $*$-representations and, in particular, we are going to generalize densely
perturbable $*$-representations, introduced in ~\cite{Dykema&Torres}.

\begin{defin}
A unital finite dimensional $*$-representation $\pi:A_1*_DA_2 \to \mathbb{B}(H)$ is
DPI if the set
\[
\Delta(\pi):=\{ u\in \mathbb{U}(\pi^{0}(D)'): \pi^{(1)}(A_1)' \cap
\Ad u (\pi^{(2)}(A_2)')=\mathbb{C}  \}
\]
is dense in $\mathbb{U}(\pi^{(0)}(D)')$. Here, $\pi^{(0)}(D)'$ denotes the commutant of $\pi^{(0)}(D)$ relatively to $\mathbb{B}(H)$.

\end{defin}

Notice that the only difference with the definition of DPI given in  ~\cite{Dykema&Torres} is that we require the unitaries to be taken
from $\mathbb{U}(\pi^{(0)}(D)')$ not only form $\mathbb{U}(\mathbb{B}(H))$.

Since translation by a unitary is a
homeomorphism we easily  get the next remark.

\begin{rmk}\label{Remark:DPIClosedByAd}
If $\pi:A_1*_DA_2 \to \mathbb{B}(H)$ is DPI and $u\in \mathbb{U}(\pi^{(0)}(D)')$, then
 $\pi^{(1)}*(\Ad u \circ \pi^{(2)})$ is DPI as well.
\end{rmk}

The following lemma is the criteria we mentioned at the beginning of the section. It really comes
from the proof of the main theorem in ~\cite{Dykema&Torres}, but with  the right modifications for the amalgamated case.

\begin{lema}\label{Lemma:CompletionCriteria}
Assume $A_1*_DA_2$ is RFD and that for all unital finite dimensional
$*$-representation  $\pi:A_1*_DA_2 \to \mathbb{B}(H)$, there is a unital
finite dimensional $*$-representation $\hat{\pi}:A_1*_DA_2 \to \mathbb{B}(\hat{H})$
such that $\pi \oplus \hat{\pi}$ is DPI. Then $A_1*_DA_2$ is primitive.
\end{lema}

\begin{proof}
We gave an sketch pointing out the main
differences for the amalgamated case.

By assumption, there is a separating family 
$(\pi_j:A_1*_DA_2 \to \mathbb{B}(H_j))_{j\geq 1} $, of finite dimensional unital $*$-representations. 
For later use in constructing an essential representation of $A_1*_DA_2$, i.e., a $*$-representation
with the property that zero is the only
compact operator in its image, we modify $(\pi_j)_{j \geq 1}$, if necessary, so that  that each 
$*$-representation is repeated infinitely many times.
 
By recursion and using our assumption, we find a 
sequence  $(\hat{\pi}_j:A_1*_DA_2 \to \mathbb{B}(\hat{H}_j))_{j\ge1}$
of finite
dimensional unital $*$-representations such that, for all  $k\geq 1$, 
$\oplus_{j=1}^{k}(\pi_j \oplus \hat{\pi}_j)$ is DPI. 
Let $\pi:=\oplus_{j \geq 1}\pi_j\oplus \hat{\pi}_j$ and $H:=\oplus_{j\geq 1}H_j\oplus \hat{H}_j$. 
To ease notation, for $k \geq 1$, let $\pi_{[k]}=\oplus_{j=1}^k\pi_j\oplus \hat{\pi}_j$.
Note that we have $\pi(A_1*_DA_2)\cap \mathbb{K}(H)=\{0\}$. Indeed, if $\pi(x)$ is compact then
$\lim_j \|(\pi_j\oplus\hat{\pi}_j )(x)\|=0$, since each representation is repeated infinitely
many times and we are considering a separating family we get $x=0$.

We will show that given any positive number  $\varepsilon$, there is a unitary $u$ in $\pi^{(0)}(D)'$ such 
that $\|u-\id_H\| < \varepsilon$ and  $\pi^{(1)}*(\Ad u \circ \pi^{(2)})$ is both irreducible and 
faithful. Note: it is crucial that $u$ lies in $\pi^{(0)}(D)'$ otherwise $\pi^{(1)}*(\Ad u \circ \pi^{(2)})$ is
not well defined. This is the main difference and the technical aspect that we have to be very careful. Fortunately, the definition of a DPI
representation takes care of this detail. Now proceed as in the proof of Theorem 5.16 in  ~\cite{Dykema&Torres}.

We will construct a sequence $(u_k,\theta_k,F_k)_{k\geq 1}$ where:
\begin{enumerate}[(i)]
\item For all $k$, $u_k$  is a unitary in $\pi_{[k]}^{(0)}(D)'$ satisfying

\begin{eqnarray}\label{ConvergenceUnitaries}
\| u_k - \id_{\oplus_{j=1}^kH_j\oplus \hat{H}_j}\| < \frac{\varepsilon}{2^{k+1}}.
\end{eqnarray}

\item Letting 
\[
u_{(j,k)}=u_j \oplus \id_{H_{j+1} \oplus \hat{H}_{j+1}} \oplus \cdots  \oplus \id_{H_k\oplus \hat{H}_k}  
\]
and
\begin{equation}\label{Eqn:ExtendingU}
U_k=u_ku_{(k-1,k)} u_{(k-2,k)}\cdots u_{(1,k)}\,,
\end{equation}
$U_k$ lies in $\pi_{[k]}^{(0)}(D)'$  and the unital $*$-representation of $A_1*_DA_2 $ onto 
$\mathbb{B}\bigl( \oplus_{j=1}^k H_j\oplus \hat{H}_j \bigr)$, given by
\begin{eqnarray}\label{RelationWithPi}
\theta_k=  \pi_{[k]}^{(1)} * (\Ad U_k \circ \pi_{[k]}^{(2)} ),
\end{eqnarray}
is irreducible.

\item $F_k $ is a finite subset of the closed unit ball of $A_1*_DA_2$   and for all $y$ in the closed 
unit  ball of  $A_1* A_2$ there is an element $x$ in $F_k$ such that
\begin{eqnarray}\label{FiniteNet}
 \| \theta_k(x)-\theta_k(y) \| < \frac{1}{2^{k+1}}\,.
\end{eqnarray}

\item If $k\geq 2$, then for any element  $x$ in the union $\cup_{j=1}^{k-1} F_j$,  we have
\begin{eqnarray}\label{RememberingSetsF}
\|  \theta_{k}(x)-(\theta_{k-1} \oplus \pi_{k}\oplus \hat{\pi}_{k})(x) \|  < \frac{1}{2^{k+1}}\,.
\end{eqnarray}

\end{enumerate}
We construct such a sequence by recursion.

\noindent
{\em Step 1: Construction of $(u_1,\theta_1,F_1)$.}
Since $\pi_{[1]}=\pi\oplus \hat{\pi}$ is DPI, there is a unitary $u_1$ in $\pi_{[1]}^{(0)}(D)'$ 
such that  $\|u_1-\id_{H\oplus \hat{H}}\|< \frac{\varepsilon}{2^2}$ and 
$\pi_{[1]}^{(1)}*\Ad u_1 \circ \pi_{[1]}^{(2)}$ is well defined and
irreducible. Hence condition (\ref{ConvergenceUnitaries})  and (\ref{RelationWithPi})  trivially  hold.
Since $H_1 \oplus \hat{H}_1$ is finite dimensional,  
there is a finite set $F_1$ contained in the closed unit  ball of $A_1*_DA_2$ satisfying condition 
(\ref{FiniteNet}).  
At this stage there is no condition (\ref{RememberingSetsF}).

\noindent
{\em Step 2: Construction of $(u_{k+1},\theta_{k+1},F_{k+1})$ from $(u_j,\theta_j,F_j)$, $1\leq  j \leq k$.}
First, we are to prove  that there exist a unitary $u_{k+1}$ in $\pi_{[k+1]}^{(0)}(D)'$ 
such that $\| u_{k+1} - \id_{\oplus_{j=1}^{k+1}H_j \oplus \hat{H}_j} \| < \frac{\varepsilon}{2^{k+2}}$,
the unital  $*$-representation of $A_1*_DA_2$ into  
$\mathbb{B}\bigl( \oplus_{j=1}^{k+1} H_j \oplus \hat{H}_j \bigr)$ defined by 
\begin{equation}\label{DefinitioRepTheta}
\theta_{k+1}:= (\theta_k \oplus \pi_{k+1}\oplus \hat{\pi}_{k+1})^{(1)}*(\Ad u_{k+1} \circ 
(\theta_k \oplus \pi_{k+1}\oplus \hat{\pi}_{k+1} )^{(2)})  
\end{equation}
 is well defined, irreducible and for any element $x$ in the union $\cup_{j=1}^k F_j$, the inequality  
$\|\theta_{k+1}(x)- (\theta_k \oplus \pi_{k+1} \oplus \hat{\pi}_{k+1})(x)   \|< \frac{1}{2^{k+1}}$, holds. 
We begin by noticing that  $U_k\oplus \id_{H_{k+1}\oplus \hat{H}_{k+1}}$ lies in
$\pi_{[k+1]}^{(0)}(D)'$. This, along with  (\ref{RelationWithPi}) gives
\[
\theta_k \oplus \pi_{k+1} \oplus \hat{\pi}_{k+1}=\pi_{[k+1]}^{(1)}*\Ad (U_k\oplus \id_{H_{k+1}\oplus \hat{H}_{k+1}}
) \circ \pi_{[k+1]}^{(2)},
\]
thus Remark  \ref{Remark:DPIClosedByAd}
assures the existence of such unitary  $u_{k+1}$.
Notice that, from construction,  conditions (\ref{ConvergenceUnitaries}) and  
(\ref{RememberingSetsF}) are satisfied. Now, it is easy to see that $u_{(j,k+1)}$ is in 
$\pi^{(0)}_{[k+1]}(D)'$ for all $j=1,\dots, k$, so we get that $U_{k+1}$ also lies in 
$\pi^{(0)}_{[k+1]}(D)'$. A consequence of (\ref{RelationWithPi}) and 
(\ref{Eqn:ExtendingU})  is
\[
\theta_{k+1}=\pi_{[k+1]}^{(1)}*(\Ad U_{k+1} \circ \pi_{[k+1]}^{(2)}).
\]
Finite dimensionality of   $\oplus_{j=1}^{k+1}H_j\oplus\hat{H}_j $  guarantees the existence of a
finite set $F_{k+1}$ contained in the closed unit ball of $A_1*_DA_2$ satisfying condition  (\ref{FiniteNet}).
This completes Step 2.

Now consider the $*$-representations  
\begin{equation}\label{DefRepSigmak}
\sigma_k=\theta_k \oplus \bigoplus_{j\geq k+1} \pi_j \oplus \hat{\pi}_j.
\end{equation}
We now show there is a unital $*$-representation of  $\sigma:A_1*_DA_2\to\mathbb{B}(H)$, 
such that for  all $x$ in $A_1*_DA_2$, $\lim_k \|  \sigma_k(x) -\sigma(x)\|=0$.
If we extend the unitaries $u_k$ to all of $H$ via 
$\tilde{u}_k=u_k  \oplus_{j\geq k+1} \id_{H_j \oplus \hat{H}_j}$, 
then we  obtain, firstly that $\tilde{u}_k\in \pi^{(0)}(D)'$ and secondly 
\begin{equation}
\sigma_k =  \pi^{(1)} *(\Ad \tilde{U}_k  \circ \pi^{(2)}) ,
\end{equation} 
where $\tilde{U}_k=\tilde{u}_k\cdots \tilde{u}_1$.
Thanks to condition (\ref{ConvergenceUnitaries}), we have
$$
\| \tilde{U}_k -\id_H \| \leq \sum_{j=1}^k\|  \tilde{u}_k - \id_H \| < \sum_{j=1}^k \frac{\varepsilon}{2^{k+1}},
$$
and for $l \geq 1$
$$
\| \tilde{U}_{k+l} - \tilde{U}_{k} \|  =\|  \tilde{u}_{k+l}\cdots \tilde{u}_{k+1}  - \id_H  \| 
\leq  \sum_{j=k+1}^{k+l} \frac{\varepsilon}{2^{j+1}}.
$$
Hence, Cauchy's criterion implies there is a unitary $u$ in $\mathbb{U}(H)$ such that the 
sequence $(\tilde{U}_k)_{k\geq 1}$ converges in norm to $u$ and $\|u-\id_H\|< \frac{\varepsilon}{2}$.
Since each $\tilde{U}_k$ commutes with the elements of $\pi^{(0)}(D)$, $u$ also commutes with
all the elements of $\pi^{(0)}(D)$. Hence the $*$-representation  
\begin{equation}\label{DefRepSigma}
\sigma=\pi^{(1)}*(\Ad u \circ \pi^{(2)})
\end{equation}
is well defined.
An standard approximation argument shows that for all $x$ in $A_1*_DA_2$,
\begin{equation}\label{PointWiseConvergenceSigma}
\lim_k \| \sigma_k(x)- \sigma(x)\| =0.
\end{equation}

Our next goal is to  show $\sigma$ is irreducible but  from this point the proof 
is identical to the proof of Theorem 5.16 in  ~\cite{Dykema&Torres}.

\end{proof}

Thus, due to the last lemma, we now focus on finding finite dimensional $*$-representations that
are irreducible.

\section{Finite dimensional C*-subalgebras in general position}

The technique that we will use, in the sense of Lemma \ref{Lemma:CompletionCriteria}, to complete a finite dimensional
$*$-representation is the one  used in ~\cite{Dykema&Torres} i.e., perturbations. For the convenience of the reader we
recall some  notation and definitions.

\begin{defin}
Let $\pi:A_1*_DA_2 \to M_N$ be a unital finite dimensional $*$-representation. For a unitary  $u $ in $\pi^{(0)}(D)'$, a perturbation of $\pi$ by $u$ is the $*$-representation given by $\pi^{(1)}*( \Ad u \circ \pi^{(2)})$. 
\end{defin}  

Notice that it is crucial that $u$ lies in the commutant of $\pi^{(0)}(D)$, otherwise the $*$-representations 
$\pi^{(1)}$ and $\Ad u \circ \pi^{(2)}$ might not agree on $D$ and $\pi^{(1)}*( \Ad u \circ \pi^{(2)})$
might not be well defined. Also notice that the irreducibility of $\pi^{(1)}*(\Ad u \circ \pi^{(2)})$ is equivalent to
\[
\pi^{(1)}(A_1)' \cap \Ad u(\pi^{(2)}(A_2)') =\mathbb{C}
\]
which, loosely speaking, is telling us that $\pi^{(1)}(A_1)'$ and $\Ad u(\pi^{(2)}(A_2)')$ are in general position.

 We  can frame the latter in the next context:  assume we have $M_N$, a simple finite dimensional C*-algebra, and let  $B_0, B_1$ and $B_2$ be  finite dimensional C*-subalgebras of $M_N$ such that $B_1$ and $B_2$ are contained in $B_0$. We are interested in finding conditions, on $B_1$ and $B_2$, such that the set 
\[
\Delta(B_1,B_2;B_0)=\{u\in \mathbb{U}(B_0): B_1\cap \Ad u(B_2)=\mathbb{C}\}
\]
is dense in $\mathbb{U}(B_0)$.

From section 4 in ~\cite{Dykema&Torres} it follows  that the set $\Delta(B_1,B_2;B_0)$ is dense if we can control de following numbers:

\begin{eqnarray*}
d(C,B_1,B_2,B_0,u):&=&\dim \mathbb{U}(B_1)-\mathbb{U}(B_1\cap C') \\
&+&  \dim \mathbb{U}(B_2)-\dim \mathbb{U}(\Ad u(B_2)\cap C') \\ 
&+& \dim \mathbb{U}(C'\cap B_0) 
\end{eqnarray*}
where all commutants are taken relative to $M_N$, $C$ is a unital, abelian proper $C^*$-subalgebra of $B_1$, with $\dim(C)\geq 2$ and $u$ is a unitary in $\mathbb{U}(B_0)$  such that
 $C$ is contained in $\Ad u(B_2)$.

Indeed, from Lemma 4.15 and Propositions 4.20 and  4.21  in ~\cite{Dykema&Torres} we  have the next  proposition.

\begin{prop}\label{Prop:DensityHardVersion}
With the same notation as above, if 
\[
\dim(\mathbb{U}(B_1))+\dim(\mathbb{U}(B_2)) \leq \dim(\mathbb{U}(B_0))
\]
and
$$
d(C,B_1,B_2,B_0,u)< \dim \mathbb{U}(B_0)
$$
for all $C$, unital abelian  proper C*-subalgebra of $B_1$, with $\dim(C)\geq 2$ and   unitary $u$ in  $\mathbb{U}(B_0)$  such that  $C$ is contained in $\Ad u(B_2)$, then $\Delta(B_1,B_2;B_0)$ is dense in $\mathbb{U}(B_0)$.
\end{prop}

To be honest,  the second assumption in Proposition \ref{Prop:DensityHardVersion} is quite demanding. Fortunately,
when $B_0$ is simple we can simplify it and this is done in the next section.

\subsection{A simple assumption} 
We specialize in the case when $B_0=M_N$. This case, for very especial instances of $B_1$ and
$B_2$, were treated in ~\cite{Dykema&Torres}.  The main purpose of this subsection is generalize Theorem 4.1 in ~\cite{Dykema&Torres} as follows.

\begin{thm}\label{Theorem:dCLesssOrEqualThanN2}
 Assume 
\[
\dim(\mathbb{U}(B_1))+\dim(\mathbb{U}(B_2)) <N^2
\]

and that the dimensions of the direct summands of $B_1$ and $B_2$ are less or
equal than $N^2/2$. Then $\Delta(B_1,B_2;M_N)$ is dense in $M_N$.

\end{thm}

The proof is elaborate, so to ease the burden we start with some notations. Since $B_1$ and $B_2$ will be fixed
for the rest of this section we  rewrite  $d(C,B_1,B_2,B_0,u)$ as
\begin{eqnarray*}
d(C,u):&=&\dim \mathbb{U}(B_1)-\dim \mathbb{U}(B_1\cap C')  \\
&+& \dim \mathbb{U}(B_2)-\dim\mathbb{U}(\Ad u(B_2)\cap C ')
\\
&+&\dim \mathbb{U}(C') 
\end{eqnarray*}

\begin{notat}\label{NotationAbelianC}
Given  $C$, a unital C*-subalgebra of $B_1$
and a unitary $u$ in $M_N$ such that $\Ad u(C)$ is contained in $B_2$, we denote 
\begin{eqnarray*}
\mu(B_1,C)&=&[a_{i,j}]_{1\leq i \leq l_1, 1 \leq j \leq l},\\
\mu(B_2,\Ad u(C))&=&[b_{i,j}]_{1\leq i \leq l_2, 1\leq j \leq l},\\
\mu(M_N,C)&=&[m(1),\dots , m(l)],\\
\mu(M_N,B_1)&=&[m_1(1),\dots , m_1(l_1)],\\
\mu(M_N,B_2)&=&[m_2(1),\dots , m_2(l_2)].
\end{eqnarray*}
where  we are taken the matrices of partial multiplicities given by the inclusions. 
\end{notat}

There are lots of algebraic relations between the entries of these matrices that we want to point out. Let $(p_1(1),\dots, p_1(l_1))$, $(p_2(1),\dots, p_2(l_2))$  denote the dimensions of the direct summands of
 $B_1$ and $B_2$ respectively. Thus, we have
\begin{eqnarray*}
\textrm{for} \quad  1 \leq i \leq l_1 : \sum_{j=1}^l a_{i,j}&=& p_1(i),\\
\textrm{for} \quad 1 \leq i \leq l_2:  \sum_{j=1}^l b_{i,j}&=& p_2(i),\\
\sum_{j=1}^l m(j)&=& N,\\
\sum_{j=1}^{l}\sum_{i=1}^{l_1}m_1(j)a_{i,j}&=&N,\\
\sum_{j=1}^{l}\sum_{i=1}^{l_2} m_1(j)b_{i,j}&=&N.
\end{eqnarray*}

Since $\mu(M_N,B_1)\mu(B_1,C)=\mu(M_N,C)=\mu(M_N,B_2)\mu(\Ad u(B_2),C )$ we also must have
\begin{eqnarray*}
\textrm{for} \quad 1 \leq j \leq l : \sum_{i=1}^{l_1} m_1(i)a_{i,j}&=&m(j),\\ 
\textrm{for} \quad  1 \leq j \leq l : \sum_{i=1}^{l_2} m_2(i)b_{i,j}&=&m(j).\\
\end{eqnarray*}

Thus, we may rewrite $d(C,u)$ as
\begin{eqnarray*}
d(C,u)&=&\sum_{i=1}^{l_1}p_1(i)^2-\sum_{i=1}^{l_1}\sum_{j=1}^la_{i,j}^2 \\
&+&\sum_{i=1}^{l_2}p_2(i)^2-\sum_{i=1}^{l_2}\sum_{j=1}^l b_{i,j^2} \\
&+& \sum_{j=1}^l m(j)^2 .
\end{eqnarray*}

With this notation,  our supposition about the dimensions of the direct summands of $B_1$ and $B_2$ translates to

\[
\max_{1\leq i \leq l_1}\{p_1(i) \}, \max_{1\leq i \leq l_2}\{p_2(i)\} \leq \frac{N}{2}.
\]

We start with an easy case for the complicated assumption of Proposition \ref{Prop:DensityHardVersion}.

\begin{lema}\label{LastUpperBound}
Assume

\[
\dim( \mathbb{U}(B_1)) +\dim( \mathbb{U}(B_2)) <N^2. 
\]
Then for any $C$, unital $C^*$-subalgebra of  $B_1$ of dimension 2 and any $u$, unitary in $M_N$ such that
$C$ is contained in $\Ad u(B_2)$, we have $d(C,u)<N^2$.
\end{lema}
\begin{proof}
We need to show that
\begin{eqnarray*}
d(C,u) &=& \dim \mathbb{U}(B_1) - \dim \mathbb{U}(B_1\cap C') \\
&+&\dim \mathbb{U}(B_2) - \dim \mathbb{U}(\Ad u(B_2)\cap C') \\
&+&\dim \mathbb{U}(C')
\end{eqnarray*}
is strictly less than $N^2$.

For $s=1,2$, $B_s$ is $*$-isomorphic to
$
\oplus_{i=1}^{l_s}M_{p_s(i)}.
$

With the notation \ref{NotationAbelianC}, and for $1\leq i \leq l_1$ or  $1\leq i \leq l_2$, we must have
\begin{eqnarray*}
a_{i,1}+a_{i,2}&=&p_1(i) , \\
b_{i,1}+b_{i,2}&=&p_2(i),\\
m(1)&=&\sum_{i=1}^{l_1}m_1(i)a_{i,1}=\sum_{i=1}^{l_2}m_2(i)b_{i,1}.
\end{eqnarray*}

Then

\begin{eqnarray*}
\dim \mathbb{U}(B_1\cap C')=2\bigg( \sum_{i=1}^{l_1}a_{i,1}^2 \bigg) -2\bigg(  \sum_{i=1}^{l_1}p_1(i)a_{i,1}\bigg)
+\sum_{i=1}^{l_1}p_1(i)^2,
\end{eqnarray*} 

\begin{eqnarray*}
\dim \mathbb{U}(\Ad u(B_2)\cap C')=2\bigg( \sum_{i=1}^{l_2}b_{i,1}^2 \bigg) -2\bigg(  \sum_{i=1}^{l_2}p_2(i)b_{i,1}\bigg)
+\sum_{i=1}^{l_2}p_2(i)^2,
\end{eqnarray*}
and 
$$
\dim \mathbb{U}(C')=2m(1)^2-2Nm(1)+N^2.
$$

Thus $d(C)<N^2$ if and only if
$$
\begin{array}{c}
m(1)^2-\bigg( \sum_{i=1}^{l_1}a_{i,1}^2 + \sum_{i=1}^{l_2}b_{i,1}^2\bigg)
\end{array}
<
\begin{array}{c} 
m(1)N-\bigg( \sum_{i=1}^{l_1}p_1(i)a_{i,1}+ \sum_{i=1}^{l_2}p_2(i)b_{i,1}\bigg).
\end{array}
$$

Now take non negative numbers $\alpha_{i}$, $1\leq i \leq l_1$, and $\beta_{i}$,
 $1\leq i \leq l_2$, such that $m_1(i)a_{i,1}=\alpha_{i}m(1)$,
 $m_2(i)b_{i,1}=\beta_{i}m(1)$ and $\sum_{i=1}^{l_1}\alpha_{i}=\sum_{i=1}^{l_2}\beta_{i}=1$.

With this change of variables the previous inequality becomes
$$
m(1)^2\bigg(  1- \sum_{i=1}^{l_1}\frac{\alpha_{i}^2}{m_1(i)^2}  -\sum_{i=1}^{l_2}\frac{\beta_{i}^2}{m_2(i)^2}   \bigg)
<
m(1)N\bigg( 1-  \sum_{i=1}^{l_1}\frac{p_1(i)\alpha_{i}}{m_1(i)N}  -\sum_{i=1}^{l_2}\frac{p_2(i)\beta_{i}}{m_2(i)N}  \bigg).
$$

We can cancel $m(1)$ because $m(1) \geq 1$. Furthermore, we may assume $m(1) \leq N/2$. Indeed, 
since $m(1)+m(2) =N$ at least one must be less or equal than $N/2$, so we may assume it is $m(1)$. Thus it suffices to show
$$
\bigg(  1- \sum_{i=1}^{l_1}\frac{\alpha_{i}^2}{m_1(i)^2}  -\sum_{i=1}^{l_2}\frac{\beta_{i}^2}{m_2(i)^2}   \bigg)
<
2\bigg( 1-  \sum_{i=1}^{l_1}\frac{p_1(i)\alpha_{i}}{m_1(i)N}  -\sum_{i=1}^{l_2}\frac{p_2(i)\beta_{i}}{m_2(i)N}  \bigg),
$$
or equivalently 
$$
0<  \sum_{i=1}^{l_1} \bigg(\frac{\alpha_{i}^2}{m_1(i)^2} -\frac{2\alpha_{i}p_1(i)}{m_1(i)N} \bigg)+\sum_{i=1}^{l_2} 
\bigg(\frac{\beta_{i}^2}{m_2(i)^2} -\frac{2\beta_{i}p_2(i)}{m_2(i)N} \bigg)+1.
$$

Completing squares we get the above inequality is equivalent to
\begin{eqnarray*}
0&<&  \sum_{i=1}^{l_1} \bigg(\frac{\alpha_{i}}{m_1(i)} -\frac{p_1(i)}{N} \bigg)^2+
\sum_{i=1}^{l_2} \bigg(\frac{\beta_{i}}{m_2(i)} -\frac{p_2(i)}{N} \bigg)^2\\
&+&1-\frac{1}{N^2}\bigg(  \sum_{i=1}^{l_1}p_1(i)^2+\sum_{i=1}^{l_2}p_2(i)^2 \bigg).
\end{eqnarray*}

But this last inequality is true by our assumption that $\dim \mathbb{U}(B_1)+\dim \mathbb{U}(B_2)< N^2$.
\end{proof}

Now the plan is to show that for any $C$, unital C*-subalgebra of $B_1$ and $u$, unitary in $M_N$, such that
$C$ is contained in $\Ad u(B_2)$, there is $C_0$,  a unital C*-subalgebra  of $C$ of dimension 2, such that  $d(C,u)\leq d(C_0,u)$.

\begin{prop}\label{dCAverageOverSubalgebras}
Assume $C$ is a unital C*-subalgebra of $B_1$ unitarily equivalent to
a C*-subalgebra of $B_2$ and $*$-isomorphic to $\mathbb{C}^l$, with $l\geq2$.
For $1\leq r \not= s \leq l$ we define $C_{(r,s)}$ as the unital C*-subalgebra of $C$
obtained by merging coordinates $r$ and $s$ in $C$ (in a given fixed order). 
Let $I=\{(r,s): 1\leq r \not= s \leq l\}$. 

Then
\[
d(C_{(r,s)}) = d(C)+2\bigg(m(r)m(s)-  \sum_{i=1}^{l_1}a_{i,r}a_{i,s} - \sum_{i=1}^{l_2}b_{i,r}b_{i,s} \bigg). 
\]
In consequence,  $d(C)\leq d(C_{(r,s)})$ for some $(r,s)\in I$ if  
\[
\sum_{i=1}^{l_1}a_{i,r}a_{i,s} +\sum_{i=1}^{l_2}b_{i,r}b_{i,s} \leq m(r)m(s).
\]

\end{prop}

\begin{proof}

  With notation \ref{NotationAbelianC}, we have
\begin{eqnarray*}
\dim \mathbb{U}(B_1\cap C_{(r,s)}')&=&\sum_{i=1}^{l_1}(a_{i,r}+a_{i,s})^2+\sum_{i=1}^{l_1} \sum_{j=1, j\not=r,s}^{l} a_{i,j}^2  \\
&=& 2\bigg( \sum_{i=1}^{l_1}a_{i,r}a_{i,s} \bigg)+ \dim \mathbb{U}(B_1\cap C').
\end{eqnarray*}

Similarly 
\begin{eqnarray*}
\dim \mathbb{U}(B_2\cap u^*C_{(r,s)}'u)&=&2\bigg( \sum_{i=1}^{l_2}b_{i,r}b_{i,s}\bigg)+\dim \mathbb{U}(B_2\cap u^*C'u), \\
\dim \mathbb{U}(C_{(r,s)}') &=& 2m(r)m(s)+ \dim \mathbb{U}(C').
\end{eqnarray*}

Thus $d(C_{(r,s)}) = d(C)+2\bigg(m(r)m(s)-  \sum_{i=1}^{l_1}a_{i,r}a_{i,s} - \sum_{i=1}^{l_2}b_{i,r}b_{i,s} \bigg) $. 

\end{proof}

\begin{prop}\label{Prop:IncreasingProperty}
With the same  notation as Proposition \ref{dCAverageOverSubalgebras}, if
\[
\max_{1\leq i \leq l_1} \{p_1(i)\},  \max_{1\leq i \leq l_2}\{p_2(i) \} \leq \frac{N}{2},
\]

then, there are  $1\leq r_0 \leq l$ and $ 1\leq s_0 \leq l$, $r_0\not=s_0$, such that
$d(C,u) \leq d(C_{(r_0,s_0)},u)$. 
\end{prop}

\begin{proof}

$\bullet$ \textbf{Case 1}. There are  $r_0$ and  $s_0$ such that
\[
\max_{1\leq i \leq l_1}\{a_{i,r_0}\}\leq \frac{m(r_0)}{2}, \quad
\max_{1\leq i \leq l_2}\{b_{i,s_0}\}\leq \frac{m(s_0)}{2}.
\]

From the hypothesis, the fact that  $\sum_{i}m_1(i)a_{i,j}=m(j)$ and $m_1(i) \geq 1$ for all
$i$, we deduce that for any $j$
\[
\sum_{i=1}^{l_1}a_{i,r}a_{i,j}\leq \frac{m(r)m(j)}{2}. 
\]

Similarly, for any $j$, 
\[
\sum_{i=1}^{l_2}b_{i,j}b_{i,s}\leq \frac{m(j)m(s)}{2}. 
\]

Thus we conclude
\[
\sum_{i=1}^{l_1}a_{i,r}a_{i,s}+\sum_{i=1}^{l_2}b_{i,r}b_{i,s} \leq m(r)m(j).
\]

$\bullet$ \textbf{Case 2}. Assume that for all $r$,
\[
\max_{1\leq i \leq l_1}\{a_{i,r}\} >  \frac{m(r)}{2},
\]

and for some $s_0$,
\[
\max_{1\leq i\leq l_2}\{ b_{i,s_0}\} \leq \frac{m(s_0)}{2}.
\]

For any $j$, let $i(j)$ be such that
\[
\max_{1\leq i \leq l_1}\{a_{i,j}\}=a_{i(j),j}.
\]

By assumption $a_{i(j),j} > m(j)/2$ for all $j$.

First we show that for all $s$ there is $r$ such that $i(s)\not= i(r)$. To prove it
we proceed by contradiction. So, we suppose there is $s$ such that for all $r$, $i(r)=i(s)$.

Let $i(s)=i_0$. Then $a_{i_0,j}>m(j)/2$ for all $j$. Hence, summing over
$j$ brings
\[
p_1(i_0)= \sum_{j}a_{i_0,j}>\frac{N}{2}
\]
 a contradiction with our assumption.

Before proceeding, we  let $\tilde{a}_{i,j}=\frac{a_{i,j}}{m(j)}$. Then, for all $j$,
\[
\sum_{i}\tilde{a}_{i,j}\leq 1.
\]

Define
\[
P=\left\{ x\in \mathbb{R}^{l_1}: \forall i, x(i) \geq 0, \quad  \sum_{i}x(i)\leq 1 \right\}.
\]
 
Then $P$ is compact,  convex and notice that $x_j:=(\tilde{a}_{i,j})_{i=1}^{l_1}$ lie in $P$.

For $1\leq i \leq l_2$ let 
\[
P_i=\{x\in P: x(i) \geq 1/2\}.
\]

By assumption $x_j\in P_{i(j)}$ and notice that $P_i$ is compact and  convex.

For a vector $x\in \mathbb{R}^{l_1}$,  take the linear functional  $F_x:\mathbb{R}^{l_1}\to \mathbb{R}$, defined by
\[
F_x(y)=\sum_{i=1}^{l_1}x(i)y(i).
\] 
Then  $\sum_{i}\tilde{a}_{i,r}\tilde{a}_{i,s}=F_{x_r}(x_s)=F_{x_s}(x_r)$.

We will show that if $i(r)\not=i(s)$, 
\[
\max_{x\in P_{i(r)}}\{F_{x_s}(x)\}\leq \frac{1}{2},
\]

which in consequence proves $\sum_i a_{i,r}a_{i,s} \leq \frac{m(r)m(s)}{2}$.

Indeed, since $F_{x_s}$ is linear and $P_{i(r)}$ is compact and convex,
\[
\max_{x\in P_{i(r)}}\{ F_{x_s}(x) \} = F_{x_s}(x_{i(r)}^*),
\]
where $x_{i(r)}^*$ is an extreme point of $P_{i(r)}$. But notice that
 $F_{x_s}(x_{i(r)}^*)=F_{x_{i(r)}^*}(x_s)$, so
$F_{x_s}(x_{i(r)}^*)\leq F_{x_{i(r)}^*}(x_{i(s)}^*)$ for some extreme point of $P_{i(s)}$. Thus
\[
\max_{x\in P_r}\{F_{x_s}(x)\} \leq \max{\langle x_{i(r)}^*, x_{i(s)}^* \rangle},
\]
where the maximum is taken over $x_{i(r)}^*$ and $x_{i(s)}^*$,  extreme points of $P_{i(r)}$ 
and $P_{i(r)}$ respectively.

If $\{e_i\}_{i=1}^{l_1}$ denotes the canonical basis, the extreme points of $P_{i(r)}$ are
 $e_{i(r)}$ and $(1/2)e_{i(r)}+(1/2)e_{i}$, $i\not= i(r)$. Thus (analysing all possibilities)
\[
\max\{\langle x_{i(r)}^*,  x_{i(s)}^* \rangle  \}=\frac{1}{2}.
\]

Lastly, take $r_0$ such that $i(r_0)\not=i(s_0)$ and get
\[
\sum_{i}a_{i,r_0}a_{i,s_0}\leq \frac{m(r_0)m(s_0)}{2}
\]
an from $\max_{1\leq i \leq l_2} \{b_{i,s_0} \} \leq \frac{m(s_0)}{2}$ we get
\[
\sum_i b_{i,r_0}b_{i,s_0} \leq \frac{m(r_0)m(s_0)}{2}
\]
and conclude $d(C_{(r_0,s_0)})\leq d(C)$.

$\bullet$ \textbf{Case 3}. Similar to case 2, interchanging roles of $a$ and $b$.

$\bullet$ \textbf{Case 4}. Assume that for all $r$
\[
\max_{1\leq  i \leq l_1}\{a_{i,r}\}>\frac{m(r)}{2}, \max_{1\leq  i \leq l_2}\{b_{i,r}\}>\frac{m(r)}{2}. 
\]

Let, for $1\leq r \leq l$, $i_A(r)$ and $i_B(r)$ be  such that
\begin{eqnarray*}
\max_{1\leq i \leq l_1}\{a_{i,r}\}=a_{i_A(r),r},\\
\max_{1\leq i \leq l_2}\{b_{i,r}\}=b_{i_B(r),r}.\\
\end{eqnarray*}

Notice that there is a unique such  $i_A(r)$, because 
$m(r)=\sum_{i=1}^{l_1}a_{i,r}m_1(i) \geq \sum_{i=1}^{l_1}a_{i,r}$. Similarly for $i_B(r)$. Thus
$i_A$ and $i_B$ define functions from $[l]=\{1,\dots, l\}$ to $[l_1]=\{1,\dots, l_1\}$ and
$[l_2]=\{1,\dots, l_2\}$, respectively. This functions in  turn  induce partitions, $\pi_A$ and
$\pi_B$, of $[l]$, where two points are in the same block of $\pi_A$ if and only if their image under
$i_A$ (respectively $i_B$) are the same.

From the assumption
\[
\max_{1\leq i \leq l_1}\{p_1(i)\}, \max_{1\leq i \leq l_2}\{p_2(i)\} \leq \frac{N}{2}
\]
we get that $|\pi_A|,|\pi_B| \geq 2$.

From the previous case, if $r,s\in [l]$ lie in different blocks of $\pi_A$ 
\[
\sum_{i}a_{i,r}a_{i,s}\leq \frac{m(r)}{2}.
\]

Thus, we only need to find $r_0$ and $s_0$ such that they lie in different blocks of $\pi_A$ AND $\pi_B$.
The latter is equivalent to show
\[
\cup_{\alpha \in \pi_A} \alpha \times \alpha \cup  \cup_{\beta \in \pi_B} \beta\times \beta 
\]

is a proper subset of $[l]\times [l]$.

Pick $\beta_0 \in \pi_B$ such that $|\beta_0|=\min_{\beta \in \pi_B}\{|\beta|\}$ and let
 $\beta_1=\cup_{\beta \in \pi_B, \beta \not= \beta_0}\beta$. We rename the 
elements of $[l]$ so that $\beta_0=\{1,\dots, b\}$,$\beta_1=\{b+1,\dots l\}$ and
$b \leq l-b $. In order to get a contradiction  we will assume that
\[
\cup_{\alpha\in \pi_A} \alpha \times \alpha \cup \beta_0\times \beta_0 \cup \beta_1\times \beta_1=[l]\times[l].
\]

Take $(x,y)\in \beta_0 \times \beta_1$ arbitrary. Since we are assuming equality, there is
$\alpha \in \pi_A$ such that $(x,y)\in \alpha \times \alpha$. Hence we conclude $\beta_0\cup \beta_1 \subseteq \alpha$.
But  $\beta_0\cup \beta_1=[l]$ and in consequence $\alpha=[l]$ and $|\pi_A|=1$, a contradiction. 

We conclude $\cup_{\alpha \in \pi_A} \alpha \times \alpha \cup \cup_{\beta \in \pi_B} \beta \times \beta $, 
is a proper subset of $[l]\times [l]$.

\end{proof}

Finally, a proof of  Theorem \ref{Theorem:dCLesssOrEqualThanN2} is at  hand.

\begin{proof}[Proof Theorem \ref{Theorem:dCLesssOrEqualThanN2} ]

From Proposition \ref{Prop:IncreasingProperty} it suffices to show that $d(C,u)<N^2$ for all $C$ of dimension 2, but this
is precisely Lemma \ref{LastUpperBound}.
\end{proof}

We end this section with an easy example where we can readily conclude the density of $\Delta(B_1,B_2;M_N)$.

\begin{coro}\label{Coro:EasyDensity}
With the previous notation, assume that all the entries of $\mu(M_N,B_s)$, $s=1,2$, are greater or equal than 2. Then $\Delta(B_1,B_2;M_N)$ is dense in $M_N$.
\end{coro}
\begin{proof}
With the previous notation our assumption implies $m_s(i) \geq 2$ for all $s=1,2$ and all $i\in [l_s]$. Since
$\sum_{i=1}^{l_s}m_s(i)p_s(i)=N$ it follows that $p_s(i) \leq N/2$ and

\[
\dim(\mathbb{U}(B_1))+\dim(\mathbb{U}(B_2)) =\sum_{i=1}^{l_1}p_1(i)^2+ \sum_{i=1}^{l_2}p_2(i)^2 \leq \frac{N^2}{4}+\frac{N^2}{4} <N^2.
\]

Thus an application of \ref{Theorem:dCLesssOrEqualThanN2} finishes the proof.
\end{proof}

\section{Primitivity}

\subsection{The linking path condition}
Before we start we want to see what could possibly prevent an
 amalgamated  full free product of the form $A_1*_DA_2$ from being primitive. A partial answer is given by Pedersen in
~\cite{Pedersen-Pullback&Pushout}, Proposition 4.7, and for the convenience of the reader we state it here. Recall that a morphism
between C*-algebras is called proper if it sends an approximate unit in the domain to an approximate unit
in the range. Since we are dealing with unital C*-algebras proper just means a unital morphism. One more note,
 Pedersen uses a categorical nomenclature, so an amalgamated full free product is a push out diagram.

\begin{prop}\label{Prop:Pedersen}
Consider a sequence of push out diagrams as below, to the left, and assume each $\alpha_n$
is a proper morphism. Then we obtain the new push out diagram below, to the right: 
\[
\begin{tabular}{c c c}
\xymatrix{
C_n \ar[r]^{\beta_n}\ar[d]^{\alpha_n} &   B_n \ar[d]^{\gamma_n} \\
A_n \ar[r]_{\delta_n} & X_n
} & \textrm{  gives} & 
\xymatrix{
\oplus_{n}C_n \ar[r]^{\oplus_{n}\beta_n}\ar[d]^{\oplus_n \alpha_n} &   \oplus_n B_n \ar[d]^{\oplus_n \gamma_n} \\
\oplus_n A_n \ar[r]_{\oplus_n \delta_n} & \oplus_n X_n
}
\end{tabular}
\]
\end{prop}

With our notation, the latter can be written as  $(\oplus A_n)*_{\oplus C_n}(\oplus B_n)\simeq \oplus_n(A_n*_{C_n}B_n)$, which
most certainly implies that $(\oplus A_n)*_{\oplus C_n}(\oplus B_n)$ is not primitive (provided there is
 more than one $C_n$). Coming back to 
$A_1*_DA_2$, for abelian $D$ with $\dim(D)\geq 2$,  if , for instance,  we could order the direct summands of $D$ in such a way that, the
matrices $\mu_1$ and $\mu_2$ are direct sum of smaller matrices, then we could apply  Pedersen's  result to conclude
$A_1*_DA_2$ is not primitive. 

For the finite dimensional case, we can use the Bratteli diagrams of the inclusions $\gamma_s$, to get
a feeling of what is happening in this type of situation. We draw the Bratteli diagrams with the following convention:
we draw points aligned in three horizontal lines, the top one correspond to the direct summands of $A_1$,
the middle one those of $D$ and the bottom line the ones coming from $A_2$.  For instance the diagram
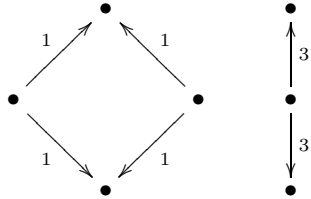
\[
\xymatrix{
 &  \bullet  &   & \bullet \\
 \bullet \ar[ur]^1 \ar[dr]_1 & & \bullet \ar[ul]_1 \ar[dl]^1 &  \bullet \ar[u]_3 \ar[d]^3 &  \\  
 &  \bullet  &   & \bullet 
 }
\]
corresponds to $(M_2\oplus M_3)*_{\mathbb{C}^3}(M_2 \oplus M_3)$ with inclusions
\[
\gamma_1(x_1,x_2,x_3)=\gamma_2(x_1,x_2,x_2)=\left[
\begin{array}{cc}
x_1 & 0 \\ 0 & x_2
\end{array}
\right] \oplus \left[  \begin{array}{ccc}
x_3 & 0 & 0 \\
 0 & x_3 & 0 \\
0 & 0 & x_3 \\
\end{array}\right]
\]
and
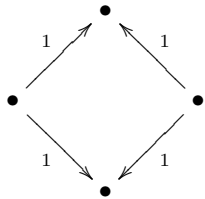
\[
\xymatrix{
 &  \bullet &   \\
\bullet \ar[ur]^1 \ar[dr]_1 &  & \bullet \ar[ul]_1 \ar[dl]^1 \\
 &  \bullet &   
 }
\]
correspond to $M_2*_{\mathbb{C}^2}M_2$ with inclusions
 
 \[
 \gamma_1(x_1,x_2)=\gamma_2(x_1,x_2)=\left[ \begin{array}{cc}
x_1 & 0 \\ 0  & x_2
\end{array} \right]
\]

Notice that in the first case we can apply Proposition \ref{Prop:Pedersen} while in the second we can't. Thus
we want to avoid cases like the first one. If, for a moment, we forget about the direction of the arrows in the Bratteli diagrams, what
is happening in the second example is that we can find a path that joints all the points of the middle line.  If we can't, then  we can apply Proposition \ref{Prop:Pedersen} and obtain non primitive C*-algebras.  Thus this is a necessary condition for primitivity. Taking into account the graphic representation we call this the Linking Path condition (LP condition for short).  Also 
notice that the LP condition is trivially full fill if the dimension of the base $D$ is one, that is why in the unital full free products
studied in ~\cite{Dykema&Torres},  this condition did not show. The formal definition is given below. 

\begin{defin}\label{Def:LPcondition}
Consider inclusions $\gamma_s:D \to A_s$, where $D$ is abelian with dimension $l_0 \geq 2$
and $l_s$ is the dimension of the center of $A_s$, $s=1,2$. We say that $A_1*_DA_2$ satisfies the LP condition if  there is  a function $c:[l_0- 1] \to [l_1]\times \{1\} \cup [l_2] \times \{2\}$ such that, for
$1\leq j  \leq l_0-1$,
\[
\mu_{c(j)[2]}(c(j)[1],j)\not=0 \not= \mu_{c(j)[2]}(c(j)[1],j+1).
\]
Here $c(j)[i]$ means the $i$-th coordinate of $c(j)$.
\end{defin}

The LP condition is a necessary but not a sufficient condition for primitivity, as the next example shows.

\begin{prop}\label{Prop:ExmapleM2M2}
Let  $\gamma_1=\gamma_2:\mathbb{C}^2 \to M_2$ denote the unital inclusions
\[
\gamma_s(x_1,x_2)=
\bigg[
\begin{array}{cc}
x_1 & 0 \\
0 & x_2 
\end{array}
\bigg].
\]

With this inclusion we have $M_2*_{\mathbb{C}^2}M_2\simeq M_2(C(\mathbb{T}))$.

\end{prop}
\begin{proof}

For convenience let $A=M_2*_{\mathbb{C}^2}M_2$ and for $i=1,2$, 
let $\iota_i:M_2 \to A$ denote the canonical
inclusion. Since $A$ contains a copy of $M_2$, we recall  that $A$ is isomorphic to $M_2(B)$, where
\[
B=\{a \in A : a\iota_1(x)=\iota_1(x)a,\quad  \textrm{for all $x\in M_2$ }  \}
\]
and an explicit isomorphism is given by
\[
\varphi(a)=[a(i,j)]_{1\leq i, j \leq 2}
\]
and  $a(i,j)=\sum_{r=1}^2\iota_1(E_{r,i})a\iota_1(E_{j,r})$.

Now, $M_2$ is generated, as algebra, by
\[
\bigg[
\begin{array}{cc}
1 & 0 \\
0 & 1
\end{array}
\bigg],
\quad
v=\bigg[
\begin{array}{cc}
1 & 0 \\
0 & -1
\end{array}
\bigg]
\quad 
\textrm{and}
\quad
u=\bigg[
\begin{array}{cc}
0 & 1 \\
1 & 0
\end{array}
\bigg].
\]

Thus $M_2(B)$ is generated as C*-algebra by $\varphi(\iota_i(u)),\varphi(\iota_i(v))$, $i=1,2$. But
a direct computation shows that $\varphi(\iota_1(u))=u$ and $\varphi(\iota_1(v))=v$. Taking into account
the amalgamation over $\mathbb{C}^2$ i.e. $\iota_1(\iota(1,-1))=\iota_2(\iota(1,-1))$,  
we get $\varphi(\iota_2(v))=\varphi(\iota_1(v))=v$. Hence $M_2(B)$ is generated by 
$1,u,v,\varphi(\iota_2(u))$. 

Next we prove
\[
\varphi(\iota_2(u))=
\bigg[
\begin{array}{cc}
0 & z \\
z^* & 0
\end{array}
\bigg]
\]
where $z\in B$ is a unitary. Indeed, the $(1,1)$ entry of $\varphi(\iota_2(u))$ is given by
$E_{1,1}\varphi(\iota_2(u))E_{1,1}$. But $E_{1,1}=\varphi(\iota_1(E_{1,1}))=\varphi(\iota_2(E_{1,1}))$, thus
$E_{1,1}\varphi(\iota_2(u))E_{1,1}=\varphi(\iota_2(E_{1,1}uE_{1,1}))=0$.
Similarly, the $(2,2)$ entry of $\varphi(\iota_2(u))$ is zero. Regarding  $(1,2)$ and $(2,1)$, a direct 
computation shows
\begin{eqnarray*}
\varphi(\iota_2(u))(1,2)&=&\iota_1(E_{1,2})\iota_2(E_{2,1})+\iota_2(E_{2,1})\iota_1(E_{1,2})=:z,\\
\varphi(\iota_2(u))(2,1)&=&\iota_2(E_{1,2})\iota_1(E_{2,1})+\iota_1(E_{2,1})\iota_2(E_{1,2})=z^*.
\end{eqnarray*}

Lastly, since $\varphi(\iota_2(u)^2)=1$ we conclude $zz^*=z^*z=1$.
\end{proof}

\subsection{Big multiplicities}

In this section we will assume that $A_1, A_2$ and $D$ are finite dimensional C*-algebras. Recall that, for $s=1,2$, 
$\gamma_s:D \to A_s$ denote a unital embedding and that $\mu_s$ denote its matrix of partial multiplicities and
$l_s$ denote the dimension of the center of $A_s$. Also, $l_0$ denotes the dimension of the center of $D$.

In this section we will prove that if we identify a large amount of portions of $D$, in $A_1$ and $A_2$, then, under the
LP condition, $A_1*_DA_2$ is primitive. To be more specific, here a large amount means bigger than 2 (see Theorem \ref{CoroPrimitivityBiggerThan2}).

\begin{lema}\label{Lemma:RFDMatricesMultiplicities}
Assume $A_1,A_2$ and $D$ are finite dimensional C*-algebras. 
$A_1*_DA_2$ is RFD if and only if, for $s=1,2$, there are column vectors $p_s\in \mathbb{Z}_+^{l_s}$ such that
$\mu_1^tp_1=\mu_2^tp_2$.
\end{lema}
\begin{proof}
By Lemma \ref{Lemma:CompresionSameMatrix}, we can assume that $D$ is abelian. Let $l_0=\dim(D)$. 
Also assume that, for $s=1,2$, $A_s$ is $*$-isomorphic to $\oplus_{i=1}^{l_s}M_{n_s(i)}$.

Firstly assume $A_1*_DA_2$ is RFD. Then, there is a  unital finite dimensional $*$-representation
$\pi:A_1*A_2 \to \mathbb{B}(H)$ such that, for $s=1,2$, $\pi^{(s)}$ is unitarily equivalent to
$\oplus_{i=1}^{l_s}\id_{M_{n_s(i)}}^{(p_s(i))}$, for some positive integers $p_s(i)$.

Let $\{e_j\}_{j=1}^n$ denote a complete set of minimal projections of $D$. Since
$\pi\circ \gamma_1=\pi \circ \gamma_2$, it follows that there is a unitary $u$ in 
$\mathbb{B}(H)$ such that, for all $j=1,\dots, l_0$,
\begin{eqnarray*}
\Ad u\bigg( \oplus_{i=1}^{l_1}\id_{M_{n_1(i)}}^{p_1(i)}(\gamma_1(e_j))  \bigg)=\oplus_{i=1}^{l_2} 
\id_{M_{n_2(i)}}^{p_2(i)}(\gamma_2(e_j)).
\end{eqnarray*}

Taking $\Tr_{\mathbb{B}(H)}$ we get, 
\begin{equation*}\label{Eqn:EqualSumProjections}
\sum_{i=1}^{l_1}\mu_1(i,j)p_1(i)=\sum_{i=1}^{l_2}\mu_2(i,j)p_2(i).
\end{equation*}
In other words, if $p_s=(p_s(1),\dots, p_s(l_s))^t$, $\mu_1^tp_1=\mu_2^tp_2$.

Now suppose that, for $s=1,2$, there are column vectors $p_s\in \mathbb{Z}_+^{l_s}$ such that
$\mu_1^tp_1=\mu_2^tp_2$. Define for $i\in [l_s]$, $\alpha_s(i)=\frac{p_s(i)n_s(i)}{\sum_{k=1}^{l_s} p_s(k)n_s(k)}$.
Then it is straightforward to check that 
 $\tau_s:=\sum_{i=1}^{l_s} \alpha_s(i)\tr_{n_s(i)}$, define (modulo a unitary conjugation) faithful traces on $A_s$ such that,
 $ \tau_1\circ \gamma_1=\tau_2 \circ \gamma_2$.
 By the results in ~\cite{Dykema&Exel&Etal-EmbeddingsFullFreeProd}, it follows that $A_1*_DA_2$ is RFD.

\end{proof}

\begin{lema}\label{Lemma:MultiplicitiesRep}
Assume $A$ is a finite dimensional $C^*$-algebra, $D \subseteq A$ is a unital abelian  
$C^*$-subalgebra of $A$ and $\pi:A \to \mathbb{B}(H)$ is a finite dimensional, unital $*$-representation. Then, for $1\leq j  \leq \dim(C(\pi(D)') )$,
$1\leq i  \leq \dim(C(\pi(A)'))$,
\[
\mu(\pi(D)',\pi(A)')(j ,i )=\mu(A,D)(i,j),
\]
where  commutants are taken relative to $\mathbb{B}(H)$ and $C(\pi(D)')$ and $C(\pi(A)')$ denotes the
center of $\pi(D)'$ and $\pi(A)'$ respectively.

\end{lema}

\begin{proof}

For simplicity take $d_0=\dim(C(\pi(D)')), d=\dim(D), a_0=\dim(C(\pi(A)')), a=\dim(C(A))$. Notice that, in general, $d_0 \leq d$
and $a_0 \leq a$, with equalities  if $\pi$ is injective. Fix and order for the direct summands of $\pi(D)'$ and $\pi(A)'$ and
for this order let $\pi(D)'[j]$ and $\pi(A)'[i]$ denote the $j$-th and $i$-th direct summands, respectively. 

For $1\leq i \leq a_0$ and for $1\leq j \leq d_0$, take $\tilde{m}_{j,i}=\mu(\pi(D)',\pi(A)')(j,i)$. Similarly,  
for $1\leq i \leq a$, $1\leq j \leq d$, let $m_{i,j}=\mu(A,D)(i,j)$. Let $\rho_j$ denote the projection form $\pi(D)'$ onto $\pi(D)'[j]$ and take $p_i$ a minimal projection on $\pi(A)'[i]$. Then, by definition,
$\tilde{m}_{j,i}=\rank(\rho_j(p_i))$. 

Assume $A$ is  $*$-isomorphic to $\oplus_{i=1}^a M_{n_i}$. We know there are unitaries $u$ in $\mathbb{B}(H)$,  $v$ in $A$ and non-negative
integers $p_i$ (some of which may be zero), such that
\[
\pi=\Ad u \circ \bigg( \oplus_{i=1}^{a}\id_{M_{n_i}}^{(p_i)}\bigg), \pi\circ \iota = \Ad uv \circ \bigg(  \oplus_{i=1}^a \oplus_{j=1}^d \id_{\mathbb{C}}^{(m_{i,j})}\bigg).
\]

It follows that
\[
\pi(D)'=\Ad uv \bigg(   C^*( E^{(i)}_{r,s}:  1\leq j \leq d_0, 1\leq r, s \leq q_j  )  \bigg),
\]
where $q_j=\sum_{i=1}^{a}m_{i,j}p_i $, $\{E_{r,s}^{(j)}\}_{1\leq r,s \leq q_j}$ is a system of matrix units in $M_{q_j}$
and for $j_1\not= j_2$, $E^{(j_1)}_{r_1,s_1}E_{r_2,s_2}^{(j_2)}=0$. Hence, for $1\leq j_0 \leq d_0$,
\[
\rho_{j_0}\bigg( \Ad uv\bigg(   \oplus_{j=1}^{d_0} \oplus_{1\leq r,s \leq q_j} z_{r,s}^{(j)} E_{r,s}^{(j)}\bigg) \bigg)=
\Ad uv \bigg( \oplus_{1\leq r,s \leq q_j} z_{r,s}^{(j_0)} E_{r,s}^{(j_0)} \bigg), 
\]
where $z_{r,s}^{(j)}$ are complex numbers.

Notice that, with no loss of generality, we can take $p_i$ as $\Ad uv(  \oplus_{j=1}^{d} \oplus_{r=1}^{m_{i,j}}E_{r,r}^{(j)} )$. Hence
$\rho_j(p_i)=\Ad uv(   \oplus_{r=1}^{m_{i,j}} E_{r,r}^{(j)} )$ and in consequence $\tilde{m}_{j,i}=\rank(\rho_j(p_i))=m_{i,j}$.

\end{proof}

\begin{thm}\label{Thm:MultiplicitiesBiggerThan2}
Let  $A_1,A_2$ and $D$ denote finite dimensional C*-algebras.
 Assume that:
\begin{enumerate}[(i)]
\item for $s\in \{1,2\}$, there are column  vectors $p_s \in \mathbb{Z}_+^{l_s}$, such that $\mu_1 ^tp_1=\mu_2^tp_2$,
\item for $s\in\{1,2\}$,  $j\in [l_0]$ and for all $i\in [l_s]$, with $\mu_s(i,j)\not=0$, $2p_s(i) \leq \sum_{i=1}^{l_s} \mu_s(i,j)p_s(i)$,
\item there is  $j_*\in [l_0], s_*\in\{1,2\}$ and $i_*\in [l_{s_*}]$ such that $\mu_{s_*}(i_*,j_*)\geq 2$,
\item $A_1*_DA_2$ satisfies the LP condition.
\end{enumerate}
Then  $A_1*_DA_2$ is primitive.
\end{thm}

\begin{proof}

By Corollary \ref{Coro:ReductionTrivialCenter}, we can assume $D$ is abelian and  together with  Lemmas  \ref{Lemma:CompresionSameMatrix}, \ref{Lemma:RFDMatricesMultiplicities} and assumption (1), $A_1*_DA_2$ is RFD. Thus, according to
Lemma \ref{Lemma:CompletionCriteria}, it suffices to show that given $\pi:A_1*_DA_2 \to \mathbb{B}(H)$, a unital
finite dimensional $*$-representation, we can find $\hat{\pi}:A_1*_DA_2 \to \mathbb{B}(\hat{H})$, another 
unital finite dimensional $*$-representation such that $\pi \oplus \hat{\pi}$ is DPI. 

Suppose that $\pi^{(s)}$ is unitarily equivalent to $\oplus_{i=1}^{l_s}\id_{M_{n_s(i)}}^{(q_s(i))}$. Write
$p_s=(p_s(1),\dots , p_s(l_s))^t$ and  take
a positive integer $k \geq 2$, such that $kp_s(i)  > q_s(i)$, for all $s\in \{1,2\}$ and all $i\in [l_s]$.  Since
$\mu_1^tp_1=\mu_2^tp_2$, there is a unital finite dimensional $*$-representation $\hat{\pi}:A_1*_DA_2 \to \mathbb{B}(\hat{H})$ such that $\hat{\pi}^{(s)}$ is unitarily equivalent to $\oplus_{i=1}^{l_s}\id_{M_{n_s(i)}}^{(kp_s(i)-q_s(i))}$. We will show that $\pi \oplus \hat{\pi}$ is DPI. 
 
Let $B_0=(\pi\oplus \hat{\pi})^{(0)}(D)'$ and for $s=1,2$, $ B_s=(\pi\oplus \hat{\pi})^{(s)}(A_s)'$. Notice that, for 
$s=1,2$, $B_s$ is $*$-isomorphic to  $\oplus_{i=1}^{l_s}M_{kp_s(i)}$ and by Lemma
\ref{Lemma:MultiplicitiesRep}, $\mu(B_0,B_s)^t=\mu_s$ . For $j\in [l_0]$ let
$$d_j:=\sum_{i=1}^{l_1}\mu_1(i,j)p_1(i)=\sum_{i=1}^{l_2}\mu_2(i,j)p_2(i).$$ 

Then $B_0$ is $*$-isomorphic to $\oplus_{j=1}^{l_0}M_{kd_j}$. Let $B_s[j]$ denote the projection from $B_s$ onto the
$j$-th direct summand of $B_0$. Then $B_s[j]$ is $*$-isomorphic to $\oplus_{i\in [l_s]: \mu_s(i,j)\not=0}M_{kp_s(i)}$.
From assumption (2), for all $j\in [l_0]$ and all $i\in [l_s]$ such that $\mu_s(i,j)\not=0$, $p_s(i)\leq d_j/2$. 

The next step is to show that 
\[
\Delta(B_1[j_*],B_2[j_*];B_0[j_*])=\{u\in \mathbb{U}(B_0[j_*]): B_1[j_*]\cap \Ad u (B_2[j_*])=\mathbb{C}\}
\]
is dense in $\mathbb{U}(B_0[j_*])$. By Theorem \ref{Theorem:dCLesssOrEqualThanN2} and assumption (2) it suffices
to prove that 
\begin{equation}\label{Eqn:IneqBigMultiplicities}
\sum_{1\leq i \leq l_1 : \mu_1(i,j_*)\not=0 }p_1(i)^2+\sum_{1\leq i \leq l_2 : \mu_2(i,j_*)\not=0 }p_2(i)^2< d_{j_*}^2.
\end{equation}
Using that  $\sum_{i=1}^{l_s}\mu_s(i,j_*)p_s(i)=d_{j_*}$, we can find  positive numbers 
$\{\beta_s(i)\}_{\{ i\in [l_s], \mu(i,j_*)\not=0\}}$, such that $\sum_{i}\beta_s(i)=1$ and $\mu_s(i,j_*)p_s(i)=\beta_s(i)d_{j_*}$. Also notice that assumption (2) implies that $\beta_s(i)/\mu_s(i,j_*)\leq 1/2$. Thus
\begin{eqnarray*}
\sum_{\substack{1\leq i \leq l_1  :\\ \mu_1(i,j_*)\not=0 }}p_1(i)^2+
\sum_{\substack{1\leq i \leq l_2 :\\  \mu_2(i,j_*)\not=0} }p_2(i)^2&=& d_{j_*}\bigg(\sum_{i}\frac{\beta_1(i)^2}{\mu_1(i,j_*)^2}+\sum_{i}\frac{\beta_2(i)^2}{\mu_2(i,j_*)^2}\bigg) \\
&\leq & \frac{d_{j_*}}{2}\bigg(\sum_{i}\frac{\beta_1(i)}{\mu_1(i,j_*)}+\sum_{i}\frac{\beta_2(i)}{\mu_2(i,j_*)}\bigg).
\end{eqnarray*}

Thus (\ref{Eqn:IneqBigMultiplicities}) holds from the assumption that $\mu_{s_*}(i_*,j_*)\geq 2$
and the fact that $\sum_i \beta_s(i)=1$.

Now, the LP condition will trigger a domino effect. Indeed, let $j_1=j_*$. By the LP condition there is $j_2 \in [l_0]$, $j_2\not=j_1$, and $p_{s_1}(i_1)$ such that the direct summand corresponding to $M_{kp_{s_1}(i_1)}$, in 
$B_{s_1}[j_1]$ also embeds in $B_0[j_2]$. If we take $u_1\in \Delta(B_1[j_1],B_2[j_1];B_0[j_1])$, and consider
$\Ad u (B_1)\cap B_2$, where $u=\prod_{j\not= j_1} id_{\mathbb{U}(B_0[j])} \times u_1$, the direct summand 
corresponding to $M_{p_{s_1}(i_1) }$ in  $B_0[j_2]$ becomes $\mathbb{C}$ so that its multiplicity in $B_0[j_1]$ grows and, since $k\geq 2$, we can apply the same reasoning again to deduce that $\Delta(B_1[j_2],B_2[j_2];B_0[j_2])$ is dense
in $\mathbb{U}(B_0[j_2])$. The LP condition guarantees that we cover all $j\in [l_0]$ so that at the end, $\Delta(B_1,B_2;B_0)$
is dense in $\mathbb{U}(B_0)$.

\end{proof}

An easy situation where all the conditions of Theorem \ref{Thm:MultiplicitiesBiggerThan2} are satisfied is  Theorem \ref{CoroPrimitivityBiggerThan2}.

\begin{proof}[Proof  of Theorem \ref{CoroPrimitivityBiggerThan2}]

Since there are faithful traces on $A_1$ and $A_2$ whose restrictions to $D$ agree, $A_1*_DA_2$ is RFD (see ~\cite{Dykema&Exel&Etal-EmbeddingsFullFreeProd} for a proof). Thus, by Lemma \ref{Lemma:RFDMatricesMultiplicities}
condition (i) in Theorem \ref{Thm:MultiplicitiesBiggerThan2} is met. The inequality in  condition (ii) of
Theorem \ref{CoroPrimitivityBiggerThan2}  is also satisfied because if
$\mu_s(i,j)\not=0$ then $\mu_s(i,j)\geq 2$. For the same reason condition (iii) is clear and since we are asking
$A_1*_DA_2$ to satisfy the LP condition this concludes the proof.

\end{proof}

\subsection{A  characterization for a class of examples}
As in the previous section,  $A_1,A_2$ and $D$
will denote finite dimensional C*-algebras. From the discussion in section 2,  for questions regarding  the
primitivity of $A_1*_DA_2$, we may assume $D$
is abelian. Recall that $l_0$ denote the dimension of $D$ and $l_1$,
$\{n_1(i)\}_{i=1}^{l_1}$ and $\mu_1$ denote, respectively, the dimension of the center of $A_1$,  the dimensions
of the direct summands of $A_1$, in some fixed order, and the matrix of partial multiplicities of the inclusion
$\gamma_1:D\to A_1$. Similarly for $A_2$. Finally we assume  $A_1*_DA_2$ is RFD.

In this section we characterize primitive C*-algebras of the form $A_1*_DA_2$, $A_1,A_2$ and $D$ finite dimensional such that
the ranks of $\mu_1$ and $\mu_2$ are both one.

\begin{rmk}
Since the rank of $\mu_s$ is one, and it is the matrix of partial multiplicities of a unital inclusion $\gamma_s:D\to A_s$, we must have that all its entries are non zero. Otherwise the inclusion would not be unital.
\end{rmk}

\begin{lema}\label{Lemma:RankOneLP}
 $A_1*_DA_2$ satisfies the LP condition.
\end{lema}
\begin{proof}
The proof will only use the fact that either $\mu_1$ or $\mu_2$ have rank one. Assume that the rank of $\mu_1$ is one. Since $\gamma_1$ is a unital inclusion we must have that all the entries
of $\mu_1$ are not zero. Take $c:[l_0-1]\to [l_1]\times \{1\} \cup [l_2]\times \{2\}$ given by  $c(j)=i_0$,
where $i_0\in [l_1] $ is arbitrary. With this function $c$, it follows from Definition \ref{Def:LPcondition} that
$A_1*_DA_2$ satisfies the LP condition.
\end{proof}

\begin{lema}\label{Lemma:IndependentQuotient}
Then there are positive integers $q_1,q_2$
such that, 
\[
\mu_1^t(q_1 1_{l_1})=\mu_2^t(q_21_{l_2}).
\]
where $1_{l}$ is the column vector of dimension $l$ whose entries are all ones.
\end{lema}
\begin{proof}
We may think $\mu_s^t$ as  a linear transformation from $\mathbb{R}^{l_1}$ to $\mathbb{R}^{l_0}$. Our assumption
on the ranks of $\mu_1$ implies that the image of $\mu_1^t$ coincides with the linear span of $\mu_1^t(1_{l_1})$
and similarly with $\mu_2$. On the other hand, since $A_1*_DA_2$ is RFD, there are $x\in \mathbb{Z}_+^{l_1}$ and 
$y\in \mathbb{Z}_{+}^{l_2}$ such that $\mu_!^t x = \mu_2^t y$. Thus there are rational numbers $r_i$, such that
$r_1 \mu_1^t(1_{l_1})=r_2\mu_2^t(l_{l_2}) $. Finally we multiply the last equality by the product of the denominators of
$r_1$ and $r_2$ to obtain $q_1$ and $q_2$.

\end{proof}

\begin{lema}\label{Lemma:UniformRep}
Given $\pi:A_1*_DA_2 \to \mathbb{B}(H)$, a unital finite dimensional $*$-representation, such that
$\pi^{(1)}$ and $\pi^{(2)}$ are injective, there are 
two positive integers $p_1$ and $p_2$  and $\hat{\pi}:A_1*_DA_2 \to \mathbb{B}(\hat{H})$, a unital finite dimensional $*$-representation, such that, $(\pi \oplus \hat{\pi})^{(s)}$ is
unitarily equivalent to $\oplus_{i=1}^{l_s}M_{n_s(i)}^{(p_s)}$, $s=1,2$. 
\end{lema}
\begin{proof}

For convenience $s$ will denote either $1$ or $2$. There are integers $Q_s(1),\dots , Q_s(l_s)$ such that $\pi^{(s)}$ is unitarily equivalent to $\oplus_{i=1}^{l_s}M_{n_s(i)}^{(Q_s(i))}$ and since $\pi^{(s)}$ is injective all the integers $Q_s(i)$ are positive. Take $q_s$  as in Lemma \ref{Lemma:IndependentQuotient}. Take  a positive integer $k_s$, such that 
$\max_{1\leq i \leq l_s}\{Q_s(i) \} <  k_sq_s$. Now consider the unital finite dimensional $*$-representation of $A_s$ given by $\rho_s=\oplus_{i=1}^{l_s}M_{n_s(i)}^{(k_sq_s -Q_s(i))}$. We would like to take the free product
$*$-representation $\rho_1*\rho_2$ but to do so we need to check that they agree on $D$. It is easy to check that the latter is equivalent to 
$$\mu_1^t(k_1p_1-Q_1(1),\dots, k_1p_1-Q_1(l_1))^t=\mu_2^t(k_2p_2-Q_2(1),\dots, k_2p_2-Q_2(l_2))^t,$$
  which
is certainly true by Lemma \ref{Lemma:IndependentQuotient} and the fact that $\pi^{(1)}$ and $\pi^{(2)}$ agree on $D$. Lastly
$\hat{\pi}=\rho_1*\rho_2$ and $p_s=k_sq_s$ satisfy the requirements of the Lemma.

\end{proof}

\begin{prop}\label{Prop:PrimitivityRankOneEasyCase}
Assume 
\begin{enumerate}[(i)]
\item there are faithful traces on $A_1$ and $A_2$ whose restrictions to $D$ agree,
\item the ranks of $\mu_1$ and $\mu_2$ are one,
\item there is $j_0\in[l_0]$ such that
\[
\frac{l_1}{\bigg( \sum_{i=1}^{l_1}\mu_1(i,j_0)\bigg)^2 }+\frac{l_1}{\bigg( \sum_{i=1}^{l_2}\mu_2(i,j_0)\bigg)^2 }<1.
\]
\end{enumerate}
Then $A_1*_DA_2$ is primitive. 
\end{prop}

\begin{proof}

Take $\rho:A_1*_DA_2 \to \mathbb{B}(H)$, a unital, finite dimensional $*$-representation.  By Lemma \ref{Lemma:CompletionCriteria}, it suffices to show there
is a unital finite dimensional $*$-representation  $\hat{\rho}:A_1*_DA_2 \to \mathbb{B}(\hat{H})$, such that
$\rho \oplus \hat{\rho}$ is DPI. 

Since our assumptions imply that $A_1*_DA_2$ is RFD,  there is 
$\sigma:A_1*_DA_2 \to \mathbb{B}(K)$, a unital finite  dimensional $*$-representation, such that
$(\rho \oplus \sigma)^{(1)}$ and $(\rho \oplus \sigma)^{(2)}$ are injective. Let
$\pi:=\rho \oplus \sigma$. From Lemma \ref{Lemma:UniformRep}, there are two integers
$p_1$, $p_2$ and $\hat{\pi}:A_1*_DA_2\to \mathbb{B}(\hat{H})$, a unital finite dimensional
$*$-representation such that, for $s=1,2$, $(\pi \oplus \hat{\pi})^{(s)}(A_s)'$ is $*$-isomorphic to $M_{p_s}^{(l_s)}$. Take
$\hat{\rho}=\sigma \oplus \hat{\pi}$. We will show that $\rho\oplus \hat{\rho}$ is DPI.

For $s=1,2$, let $B_s=(\pi \oplus \hat{\pi}) ^{(s)}(A_s)'$ and let $B_0=\pi^{(0)}(D)'$. From construction,
$B_s$ is $*$-isomorphic to $M_{p_s}^{(l_s)}$. Also notice that if we let $B_0[j]$ denote the $j$-th direct summand
of $B_0$, then $B_0[j]$ is $*$-isomorphic to $M_{d_j}$ where $d_j=\sum_{i=1}^{l_1}\mu_1(i,j)p_1=\sum_{i=1}^{l_2}\mu_2(i,j)p_2$.

On the other hand, for $s=1,2$, let $B_s[j]$ be the projection of $B_s$ onto $B_0[j]$. Since $\mu_s$ have rank one, for 
all $j\in [l_0]$, $B_s[j]$ is $*$-isomorphic to $M_{p_s}^{(l_s)}$.  The next step is to show that $B_1[j_0],B_2[j_0]$
and $B_0[j_0]$, satisfy the  hypothesis of Theorem \ref{Theorem:dCLesssOrEqualThanN2}. Firstly, we need
to show that $p_s \leq d_{j_0}/2$ which is equivalent to $2\leq \sum_{i=1}^{l_s}\mu_s(i,j_0)$. But
since
\[
\frac{l_s}{\bigg( \sum_{i=1}^{l_s}\mu_s(i,j_0)\bigg)^2}<1
\]
we get $\sqrt{2} \leq \sqrt{l_s+1} \leq \sum_{i=1}^{l_s} \mu_s(i,j_0)$. Since $\sum_{i=1}^{l_s}\mu_s(i,j_0)$ is
a positive integer we get $2 \leq \sum_{i=1}^{l_s}\mu_s(i,j_0)$.

The last condition to apply Theorem \ref{Theorem:dCLesssOrEqualThanN2} is to show that $\dim(\mathbb{U}(B_1[j])+\dim(\mathbb{U}(B_2[j]))< \dim(\mathbb{U}(B_0[j]))$. But
\begin{eqnarray*}
\dim(\mathbb{U}(B_1[j_0])+\dim(\mathbb{U}(B_2[j_0]))&=& l_1p_1^2+l_2p_2^2 \\
&=& d_{j_0}^2\bigg(  \frac{l_1}{(\sum_{i=1}^{l_1}\mu_1(i,j_0))^2} \\
&+& \frac{l_1}{(\sum_{i=1}^{l_2}\mu_2(i,j_0))^2} \bigg) \\
& <  & d_{j_0}^2.
\end{eqnarray*}
We conclude that $\Delta(B_1[j_0],B_2[j_0];B_0[j_0])$ is dense in $\mathbb{U}(B_0[j_0])$.

Now, to show that $\pi$ is DPI we will show that 
\[
\prod_{\substack{j=1 \\ j\not= j_0}}^{l_0} \mathbb{U}(B_0[j])\times \Delta(B_1[j_0],B_2[j_0]; B_0[j_0]) \subseteq \Delta
(B_1,B_2;B_0).
\]
That is, we only need to perturb the $j_0$ coordinate by a unitary in $\Delta(B_1[j_0],B_2[j_0];B_0[j_0])$. This follows
from the fact that $B_s$ is $*$-isomorphic to $M_{p_s}^{(l_s)}$ and that, for all $j\in [l_0]$, $B_s[j]$ is
also $*$-isomorphic to $M_{p_s}^{(l_s)}$. We conclude $\Delta
(B_1,B_2;B_0)$ is dense in $\mathbb{U}(B_0)$. Hence $\pi \oplus \hat{\pi}=\rho\oplus \hat{\rho}$ is DPI. 
\end{proof}

\begin{rmk}\label{Remark:CasesRankOne}
Proposition \ref{Prop:PrimitivityRankOneEasyCase} guarantees that the algebra $A_1*_DA_2$ is primitive except in the
following cases:
\begin{enumerate}[(i)]
\item $l_2 \geq 2$, $l_1=1$, $\mu_1(i,j)=1$ for all $j\in [l_0]$. Notice that in this case, necessarily $A_1\simeq M_{l_0}$. 

\item $l_1 \geq 2$, $l_2=1$, $\mu_2(i,j)=1$ for all $j\in [l_0]$. Notice that in this case, necessarily $A_2\simeq M_{l_0}$. 

\item $l_1 = 2,l_2 = 2 $, $\mu_1(i,j)=1, \mu_2(i,j)=1$ for all $i$ and $j$.

\item $l_1=l_2=1$ and for all $j\in [l_0]$, $\mu_1(1,j)=1$ or $\mu_2(1,j)=1$.
\end{enumerate}
\end{rmk}

Cases (1) and (2) are symmetric and the next proposition deals with case (1).

\begin{prop}\label{Prop:PrimitivityRankOneMatrixAlgCase}
Assume 
\begin{enumerate}[(i)]
\item $l_1=1$, $\mu_1(i,j)=1$ for all $j\in [l_0]$,
\item $l_2 \geq 2$,
\item there are faithful traces on $M_{l_0}$ and $A_2$ whose restrictions to $D$ agree,
\item the rank $\mu_2$ is one,
\item \begin{equation}\label{Eqn:InequalityPrimitiveCaseRankOneMatrixAlg}
1+\sum_{j=1}^{l_0} \frac{1}{\sum_{i=1}^{l_2} \mu_2(i,j)}< l_0.
\end{equation}
\end{enumerate}
 Then $A_1*_{D}A_2$ is primitive. 
\end{prop}
\begin{proof}

Let $\pi:M_{l_0}*_{\mathbb{C}^{l_0}}A_2 \to \mathbb{B}(H)$ be a unital finite dimensional $*$-representation. From
Lemma \ref{Lemma:CompletionCriteria}, we will finish if we manage to find $\hat{\pi}:M_{l_0}*_{\mathbb{C}^{l_0}}A_2 \to \mathbb{B}(\hat{H})$,
a unital finite dimensional $*$-representation such that $\pi \oplus \hat{\pi}$ is DPI. As in Proposition
\ref{Prop:PrimitivityRankOneEasyCase}, we might assume $\pi^{(1)}$ and $\pi^{(2)}$ are injective and then
Lemma \ref{Lemma:UniformRep} assures the existence of two integers $p_1$ and $p_2$, such that, for $s=1,2$,
$(\pi \oplus \hat{\pi})^{(s)}(A_s)'$ is $*$-isomorphic to $M_{p_s}^{(l_s)}$. We will show that
$\pi \oplus \hat{\pi}$ is DPI.

Proving that $\pi \oplus \hat{\pi}$ is DPI is not straightforward. Recall that in order to prove 
$\pi \oplus \hat{\pi}$ is  DPI we need to show 
\[
\Delta(B_1,B_2;B_0)=\{u\in \mathbb{U}(B_0): B_1\cap \Ad u(B_2)=\mathbb{C}\}
\]
is dense in $\mathbb{U}(B_0)$, where $B_0=\pi^{(0)}(B_0)'$, $B_1=(\pi \oplus \hat{\pi})(M_{l_0})'$ 
and $B_2=(\pi \oplus \hat{\pi})(A_2)'$.

 Well, the first step is to replace $B_2$. Let 
 $\tilde{A}_2$ denote the maximal abelian subalgebra of $A_2$ with the property that
 $ \gamma_2(D) \subseteq \tilde{A}_2$ and let $\tilde{B}_2=\pi^{(2)}(\tilde{A}_2)'$. The contention
 $\tilde{A}_2 \subseteq A_2$ implies $\Delta(B_1,\tilde{B}_2;B_0) \subseteq \Delta(B_1,B_2:B_0)$. Due to calculations, 
it is going to be easier to show the density of $\Delta(B_1,\tilde{B}_2;B_0)$  than that of $\Delta(B_1,B_2;B_0)$. 
  
The calculation that we just mentioned come from Proposition \ref{Prop:DensityHardVersion} which we will
use to prove $\Delta(B_1,\tilde{B}_2;B_0)$ is dense is $\mathbb{U}(B_0)$. Thus,
let $C$ be a unital abelian  proper C*-subalgebra of $B_1$, with $\dim(C)\geq 2$ and  let   $u$ be in  $\mathbb{U}(B_0)$  such that  $C$ is contained in $\Ad u(\tilde{B}_2)$. According to Proposition \ref{Prop:DensityHardVersion} we need to show that
\begin{equation}\label{Eqn:DensityEasy}
\dim(\mathbb{U}(B_1))+\dim(\mathbb{U}(\tilde{B}_2))\leq \mathbb{U}(B_0)
\end{equation}
and

\begin{equation}\label{Eqn:DensityHard}
\left. \begin{array}{ll}
\hspace{0.4cm} \dim(\mathbb{U}(B_1)) - \dim(\mathbb{U}(B_1\cap C'))  &   \\
  +\dim(\mathbb{U}(\tilde{B}_2))-  \dim(\mathbb{U}(\Ad u(\tilde{B}_2)\cap C')) &   <  \dim(\mathbb{U}(B_0))\\
  +\dim(\mathbb{U}(B_0\cap C'))  &  \\
\end{array}
\right\}
\end{equation}
where $C'$ is the commutant relative to $\mathbb{B}(H\oplus \hat{H})$.

We begin by showing  that  (\ref{Eqn:DensityEasy})  holds. Firstly, recall that
$A_2\simeq \oplus_{i=1}^{l_2}M_{n_2(i)}$. Hence $\tilde{A}_2\simeq \oplus_{i=1}^{l_2}\oplus_{k=1}^{n_2(i)}\mathbb{C}$.
Let $\tilde{l}_2=\sum_{i=1}^{l_2}n_2(i)$ denote the dimension of $\tilde{A}_2$.

On the other hand, the election of
 $\hat{\pi}$ implies that $B_1 \simeq M_{p_1}$,  $B_2\simeq M_{p_2}^{(l_2)}$   and   $B_0\simeq \oplus_{j=1}^{l_0}M_{d_j}$ 
where $d_j=\sum_{i=1}^{l_2}\mu_2(i,j)p_2=\mu_1(1,j)p_1$. Since we are assuming $\mu_1(1,j)=1$, for all
$j\in [l_0]$, we obtain  $d_j=p_1$ and $(\sum_{i=1}^{l_2}\mu_2(i,j))p_2=p_1$, for all $j\in [l_0]$. It follows that
$\tilde{B}_2\simeq M_{p_2}^{(\tilde{l}_2)}$.

Hence $\dim(\mathbb{U}(B_0))=l_0p_1^2$, $\dim(\mathbb{U}(B_1))=p_1^2$ and $\dim(\mathbb{U}(\tilde{B}_2))=\tilde{l}_2p_2$. Now notice that $\tilde{l}_2= \sum_{j=1}^{l_0}\sum_{i=1}^{l_2}\mu_2(i,j)$. Thus, taking into account $(\sum_{i=1}^{l_2}\mu_2(i,j))p_2=p_1$, we deduce 
\begin{align*}
\dim(\mathbb{U}(B_1))+\dim(\mathbb{U}(\tilde{B}_2))&=&p_1^2+ \sum_{j=1}^{l_0} \bigg( \sum_{i=1}^{l_2}\mu_2(i,j)\bigg)p_2^2 \\
&=& p_1^2\bigg(  1+ \sum_{j=1}^{l_0} \frac{1}{\sum_{i=1}^{l_2}\mu_2(i,j)} \bigg).
\end{align*}

Therefore  (\ref{Eqn:DensityEasy}) holds because of (\ref{Eqn:InequalityPrimitiveCaseRankOneMatrixAlg}). 

To prove  (\ref{Eqn:DensityHard}) we need to work a little bit harder. Let $l=\dim(C)$ and let
\begin{eqnarray*}
\mu(B_1,C)&=&[a_{r}]_{1 \leq r \leq l},\\
\mu(\tilde{B}_2,u^*Cu)&=&[b_{i,r}]_{1\leq i \leq \tilde{l}_2, 1\leq r \leq l},\\
\mu(B_0,C)&=&[m_{j,r}]_{1\leq j \leq l_0, 1\leq r \leq l}.\\
\end{eqnarray*}
With this notation we need to show
\[
p_1^2-\sum_{ r=1}^{l} a_{1,r}^2+l_2p_2^2-\sum_{i=1}^{\tilde{l}_2}\sum_{r=1}^lb_{i,r}^2+\sum_{j=1}^{l_0}\sum_{r=1}^l m_{j,r}^2 < l_0p_1^2.
\]
By the proof of (\ref{Eqn:InequalityPrimitiveCaseRankOneMatrixAlg}), the latter is equivalent to
\begin{equation}\label{Eqn:DensityHardSimplified}
\sum_{ \substack{1 \leq j \leq l_0 \\ 1 \leq r \leq l}} m_{j,r}^2-\sum_{1\leq r \leq l} a_{1,r}^2-
\sum_{ \substack{1\leq i \leq \tilde{l}_2 \\ 1 \leq r \leq l}}b_{i,r}^2 < p_1^2\bigg(l_0-1-\sum_{j=1}^{l_0} \frac{1}{\sum_{i=1}^{l_2}\mu_2(i,j)} \bigg).
\end{equation}

At this point, we want to highlight some relations that will help us to prove (\ref{Eqn:DensityHardSimplified}). Let
$\tilde{\mu}_2$ denote the matrix of partial multiplicities of the inclusion $\gamma(D)\hookrightarrow \tilde{A}_2$. From
the identities $\mu(B_0,B_1)\mu(B_1,C)=\mu(B_0,C)=\mu(B_0,\tilde{B}_2)\mu(\tilde{B}_2,C)$ and the fact that
$\mu_1^t=\mu(B_0,B_1), \tilde{\mu}_2^t=\mu(B_0,\tilde{B}_2)$ (Lemma \ref{Lemma:MultiplicitiesRep}), we deduce that:

\begin{eqnarray*}
\textrm{for all $r\in [l]$ and all $j\in [l_0]$ ,   $a_{r}=m_{j,r}$, } \\
\textrm{for all $r\in [l]$ and all  $j\in [l_0]$,  $\sum_{i=1}^{\tilde{l}_2}\tilde{\mu}_2(i,j)b_{i,r}=m_{j,r}$.}
\end{eqnarray*}

Since $C$ is abelian we also must have:
\[
\textrm{for all $j\in [l_0]$, $\sum_{r}m_{j,r}=p_1$ }.
\]

Now we proceed with (\ref{Eqn:DensityHardSimplified}). From Cauchy-Schwartz's inequality
\[
\sum_{\substack{1\leq i \leq \tilde{l}_2 \\ \tilde{\mu}_2(i,j) \not=0} } b_{i,r}^2 \geq \frac{m_{j,r}^2}{\sum_{i=1}^{\tilde{l}_2}\tilde{\mu}_2(i,j)^2}.
\]
But notice that  $\tilde{\mu}_2(i,j)$ is either 0 or 1 and even more, from the selection of $\tilde{A}_2$, we have that $\sum_{i=1}^{\tilde{l}_2}\tilde{\mu}_2(i,j)=\sum_{i=1}^{l_2}\mu_2(i,j)$. Hence, for all $r\in [l]$ and all $j\in[l_0]$
\[
\sum_{\substack{1\leq i \leq \tilde{l}_2 \\ \tilde{\mu}_2(i,j) \not=0} } b_{i,r}^2 \geq \frac{m_{j,r}^2}{\sum_{i=1}^{l_2} \mu_2(i,j)}.
\]
Here, it is important to realize that the sets $\{i\in[\tilde{l}_2]: \tilde{\mu}_2(i,j)\not= 0\}$, for $j\in [l]$, form a partition of $[\tilde{l}_2]$. Thus
\[
\sum_{i=1}^{\tilde{l}_2}b_{i,r}^2 \geq \sum_{j=1}^{l_0}\frac{m_{j,r}^2}{\sum_{i=1}^{l_2}\mu_2(i,j)}  .
\]

With all this under consideration
\begin{eqnarray*}
\sum_{ r=1}^{ l} \sum_{ j=1}^{ l_0} m_{j,r}^2-\sum_{ r=1}^{l} a_{r}^2-
 \sum_{  r=1}^{l}\sum_{  i=1}^{\tilde{l}_2}b_{i,r}^2  \leq  \\
 \sum_{r=1} ^{l}\sum_{ j =1}^{ l_0} m_{j,r}^2-\sum_{ r=1}^{l} a_{r}^2-
\sum_{ r=1}^{l} \sum_{ j=1}^{l_0}  \frac{m_{j,r}^2}{\sum_{i=1}^{l_2}\mu_2(i,j)} .
\end{eqnarray*}

Next pick $j_1\in [l_0]$ such that  $\min_{j\in [l_0]}\{ \sum_{r=1}^l m_{j,r}^2 \}=\sum_{r=1}^l m_{j_1,r}^2$ and take
$j_2\in [l_0]$ with $j_1\not= j_2$ (recall that $l_0 \geq 2$). Using that $a_r=m_{j_2,r}$, for all $r$, we obtain
\begin{eqnarray*}
\sum_{r=1} ^{l}\sum_{ j =1}^{ l_0} m_{j,r}^2-\sum_{ r=1}^{l} a_{r}^2-
\sum_{ r=1}^{l} \sum_{ j=1}^{l_0}  \frac{m_{j,r}^2}{\sum_{i=1}^{l_2}\mu_2(i,j)} \leq  \\
  \sum_{\substack{ 1\leq j \leq l_0 \\
  j\not= j_1}}  \sum_{r=1} ^{l}  m_{j,r}^2-\sum_{\substack{ 1\leq j \leq l_0 \\ j\not= j_1,j_2 }}\sum_{r=1}^l \frac{m_{j,r}^2}{\sum_{i=1}^{l_2}\mu_2(i,j)}  -
  \sum_{r=1}^l \frac{m_{j_1,r}^2}{\sum_{i=1}^{l_2}\mu_2(i,j_2)} .
\end{eqnarray*}

If we define $g:[l_0]\setminus \{j_2 \} \to \mathbb{Q}$ by $g(j)=\frac{1}{\sum_{i=1}^{l_2}\mu_2(i,j)}$ for
$j\not= j_1$ and $g(j_1)=\frac{1}{\sum_{i=1}^{l_2}\mu_2(i,j_1)}+\frac{1}{\sum_{i=1}^{l_2}\mu_2(i,j_2)}$, we can rewrite the last expression as
\begin{eqnarray*}
  \sum_{\substack{ 1\leq j \leq l_0 \\
  j\not= j_1}}  \sum_{r=1} ^{l}  m_{j,r}^2-\sum_{\substack{ 1\leq j \leq l_0 \\ j\not= j_1,j_2 }}\sum_{r=1}^l  g(j)m_{j,r}^2
=\sum_{\substack{ 1\leq j \leq l_0 \\
  j\not= j_1}} \bigg(\sum_{r=1}^lm_{j,r}^2 \bigg)(1-g(j)).
 \end{eqnarray*}

Lastly, since $1-g(j)>0$ and $\sum_{r=1}^lm_{j,r}=p_1$, for all $j\in [l_0]$ and $l\geq 2$, we deduce
\[
\sum_{\substack{ 1\leq j \leq l_0 \\
  j\not= j_1}} \bigg(\sum_{r=1}^lm_{j,r}^2 \bigg)(1-g(j))< p_1^2\bigg(  l_0-1-\sum_{j=1}^{l_0}\frac{1}{\sum_{i=1}^l \mu_2(i,j)} \bigg)
\]
and therefore (\ref{Eqn:DensityHardSimplified}) holds.
\end{proof}

From Proposition \ref{Prop:PrimitivityRankOneMatrixAlgCase} we obtain that the algebras of the form
 $M_{l_0}*_{\mathbb{C}^{l_0}}A_2$, where $\mu_2$ has rank one, are
primitive, except for $M_2*_{\mathbb{C}^2}(M_2\oplus M_2)$. But this and case (3) in Remark \ref{Remark:CasesRankOne}  are covered
by the next proposition.

\begin{prop}\label{Prop:PrimitivityOnlyOnes}
Assume that, for $s=1,2$,  $\mu_s(i,j)=1$ for all $i\in [l_s]$ and all $j\in [l_0]$.
If $l_0 \geq 2$ and one of $l_1$ or $l_2$ is bigger than 1, then $A_1*_{\mathbb{C}^{l_0}}A_2$ is primitive.
\end{prop}
\begin{proof}

Let $\pi:A_1*_{\mathbb{C}^{l_0}}A_2 \to \mathbb{B}(H)$ be a unital finite dimensional $*$-representation.We might assume $\pi^{(1)}$ and $\pi^{(2)}$ are injective . By 
Lemma \ref{Lemma:UniformRep} there are two integers $p_1$ and $p_2$, such that, for $s=1,2$,
$(\pi \oplus \hat{\pi})^{(s)}(A_s)'$ is $*$-isomorphic to $M_{p_s}^{(l_s)}$. We will show that
$\pi \oplus \hat{\pi}$ is DPI.

Let $B_0=(\pi \oplus \hat{\pi})(D)'$, $B_s=(\pi \oplus \hat{\pi})(A_s)'$. Then
$B_s$ is $*$-isomorphic to $M_{p_s}^{(l_s)}$ and $B_0$ is $*$-isomorphic to $\oplus_{j=1}^{l_0}M_{d_j}$
where $d_j=\sum_{i=1}^{l_1}\mu_1(i,j)p_1=\sum_{i=1}^{l_2}\mu_2(i,j)p_2$. Since $\mu_s(i,j)=1$ for all $i$ and
$j$ we get $d_j=l_1p_1=l_2p_2$.

Again, to show $\pi \oplus \hat{\pi}$ is DPI we are going to show,  according to Proposition \ref{Prop:DensityHardVersion} , that

\begin{equation}
\dim(\mathbb{U}(B_1))+\dim(\mathbb{U}(B_2))\leq \mathbb{U}(B_0)
\end{equation}
and
\begin{equation}
\left. \begin{array}{ll}
\hspace{0.4cm} \dim(\mathbb{U}(B_1)) - \dim(\mathbb{U}(B_1\cap C'))  &   \\
  +\dim(\mathbb{U}(B_2))-  \dim(\mathbb{U}(\Ad u(B_2)\cap C')) &   <  \dim(\mathbb{U}(B_0))\\
  +\dim(\mathbb{U}(B_0\cap C'))  &  \\
\end{array}
\right\}
\end{equation}
where $u\in \mathbb{U}(B_0)$ is such that $C\subseteq Ad u(B_2)$ and  $C'$ is the commutant relative to $\mathbb{B}(H\oplus \hat{H})$.

The first inequality  becomes $l_1p_1^2+l_2p_2^2< \sum_{j=1}^{l_0}d_j^2$. But $d_j=l_1p_1=l_2p_2$, so this inequality holds
since $l_0 \geq 2$ and because one of $l_1$ or $l_2$ is bigger than 1. 

To prove the second inequality let $l=\dim(C)$ and let
\begin{eqnarray*}
\mu(B_1,C)&=&[a_{i,r}]_{1\leq i \leq l_1, 1 \leq r \leq l},\\
\mu(B_2,u^*Cu)&=&[b_{i,r}]_{1\leq i \leq l_2, 1\leq r \leq l},\\
\mu(B_0,C)&=&[m_{j,r}]_{1\leq j \leq l_0, 1\leq r \leq l}.\\
\end{eqnarray*}

Thus we need to show
\begin{equation}\label{Eqn:AllOnes}
\sum_{j=1}^{l_0}\sum_{r=1}^{l}m_{j,r}^2-\sum_{i=1}^{l_1}\sum_{r=1}^l a_{i,r}^2-\sum_{i=1}^{l_2} \sum_{r=1}^l b_{i,r}^2<\sum_{j=1}^{l_0}d_j-l_1p_1^2-l_2p_2^2.
\end{equation}

Since $\mu_s(i,j)=1$ for all $i$, $j$ and $s$, it follows that, for all $j\in [l_0]$:
\[
\sum_{i=1}^{l_1}a_{i,r}=m_{j,r}=\sum_{i=1}^{l_2}b_{j,r},
\]
and in consequence, for fixed $j_1,j_2 \in [l_0]$, $j_1\not= j_2$, we have
\[
\sum_{i=1}^{l_1}a_{i,r}^2=m_{j_1,r}^2-\sum_{\substack{1\leq i_1, i_2 \leq l_1 \\ i_1\not= i_1}}a_{i_1,r}a_{i_2,r},
\]
\[
\sum_{i=1}^{l_2}b_{i,r}^2=m_{j_2,r}^2-\sum_{\substack{1\leq i_1, i_2 \leq l_2 \\ i_1\not= i_1}}b_{i_1,r}b_{i_2,r}.
\]

Thus, the left-hand side  of  (\ref{Eqn:AllOnes}) equals
\[
 \sum_{ \substack{1\leq  j \leq l_0 \\ j \not= j_1, j_2}} \sum_{r=1}^l m_{j,r}^2+ \sum_{\substack{1\leq i_1, i_2 \leq l_2 \\ i_1\not= i_1}} \sum_{r=1}^la_{i_1,r} a_{i_2,r}+\sum_{\substack{1\leq i_1, i_2 \leq l_2 \\ i_1\not= i_1}}\sum_{r=1}^l b_{i_1,r} b_{i_2,r}.
\]

Next, we use that, for all $i$, $\sum_{r=1}^l a_{i,r}=p_1$, $\sum_{r=1}^lb_{i,r}=p_2$ and Cauchy inequality to obtain
\[
\sum_{r=1}^l a_{i_1,j}a_{i_2,r} < p_1^2,  \quad \sum_{r=1}^l b_{i_1,j}b_{i_2,r} < p_2^2.
\]

Here, it is important to notice that we have a strict  inequality  since $l\geq 2$.

Hence, the left-hand side of (\ref{Eqn:AllOnes}) is bounded above by
\[
 \sum_{j=1}^{l_0}\sum_{r=1}^{l}m_{j,r}^2-\sum_{i=1}^{l_1}\sum_{r=1}^l a_{i,r}^2-\sum_{i=1}^{l_2} \sum_{r=1}^l b_{i,r}^2 < \sum_{ \substack{1\leq  j \leq l_0 \\ j \not= j_1, j_2}} \sum_{r=1}^l m_{j,r}^2+l_1(l_1-1)p_1^2+l_2(l_2-1)p_2^2.
\]
Then, to prove (\ref{Eqn:AllOnes}), it suffices to show
\[
 \sum_{ \substack{1\leq  j \leq l_0 \\ j \not= j_1, j_2}} \sum_{r=1}^l m_{j,r}^2 \leq  \sum_{j=1}^{l_0}d_j^2-(l_1p_1)^2-(l_2p_2)^2=\sum_{\substack{1\leq  j \leq l_0 \\ j \not= j_1, j_2}} d_j^2
\]
which follows from the fact that $l_0 \geq 2$ and $\sum_{r=1}^l m_{j,r}=d_j$.
\end{proof}

Notice that previous proposition if false for $l_0=1$, because in this case we obtain $\mathbb{C}^2*_{\mathbb{C}}\mathbb{C}^2$,
which is known to have a non trivial center and hence not primitive.

\begin{prop}\label{MatrixMatrix}
Assume:
\begin{enumerate}[(i)]
\item $n_1,n_2 \geq 2$, 
\item $D$ is abelian and $\dim(D)=l_0 \geq 2$,
\item either 
\begin{enumerate}[(a)]
\item there is $s_1\in \{1,2\}$ and $j_1\in [l_0]$ such that $\mu_{s_1}(1,j_1) \geq 2$, 
\item $l_0 \geq 3$.
\end{enumerate}
\end{enumerate}
Then any
$\pi:M_{n_1}*_DM_{n_2}\to \mathbb{B}(H)$, unital finite dimensional $*$-representation is DPI.

\end{prop}

\begin{proof}

First of all we notice that $M_{n_1}*_DM_{n_2}$ always satisfies the LP condition.

Let $B_0=\pi^{(0)}(D)'$  and for $s=1,2$, let $B_s=\pi^{(s)}(M_{n_s})'$. Then $B_s$ 
is $*$-isomorphic to $M_{p_s}$ and $B_0$ is $*$-isomorphic to $\oplus_{j=1}^{l_0}M_{d_j}$
where $d_j=\mu_1(1,j)p_1=\mu_2(1,j)p_2$. To simplify we denote $\mu_s(j)=\mu_s(1,j)$.

To show $\pi$ is DPI we will use Proposition \ref{Prop:DensityHardVersion}. Thus, we need to show

\begin{equation}\label{Eqn:MatrixMatrixEasy}
\dim(\mathbb{U}(B_1))+\dim(\mathbb{U}(B_2))\leq \mathbb{U}(B_0)
\end{equation}
and
\begin{equation}\label{Eqn:MatrixMatrixHard}
\left. 
\begin{array}{ll}
\hspace{0.4cm} \dim(\mathbb{U}(B_1)) - \dim(\mathbb{U}(B_1\cap C'))  &   \\
  +\dim(\mathbb{U}(B_2))-  \dim(\mathbb{U}(\Ad u(B_2)\cap C')) &   <  \dim(\mathbb{U}(B_0))\\
  +\dim(\mathbb{U}(B_0\cap C'))  &  \\
\end{array}
\right\}
\end{equation}
where $C$ is an abelian unital C*-subalgebra of $B_1$ with $\dim(C)=l \geq 2$, $C'$ is the commutant relative to $\mathbb{B}(H\oplus \hat{H})$ and $u\in \mathbb{U}(B_0)$ is such that $C\subseteq Ad u(B_2)$.

Firstly, we prove inequality (\ref{Eqn:MatrixMatrixEasy}).  Notice that
$\dim(\mathbb{U}(B_1))+\dim(\mathbb{U}(B_2))=p_1^2+p_2^2$
 and $\dim(\mathbb{U}(B_0))=\sum_{j=1}^{l_0}d_j^2$. Taking
 into account that  $d_j=\mu_1(j)p_1=\mu_2(j)p_2$ is clear that if either $l_0\geq 3$ or  $\mu_{s_1}(j_1) \geq 2$,
(\ref{Eqn:MatrixMatrixEasy}) holds.

To prove (\ref{Eqn:MatrixMatrixHard}) write

\begin{eqnarray*}
\mu(B_1,C)&=&[a_{r}]_{1\leq r \leq l },\\
\mu(B_2,\Ad u(C))&=&[b_{r}]_{1\leq r \leq l},\\
\mu(B_0,C)&=&[m_{j,r}]_{1\leq j \leq l_0, 1\leq r \leq  l}.
\end{eqnarray*}

Thus (\ref{Eqn:MatrixMatrixHard}) becomes

\[
\sum_{j=1}^{l_0}\sum_{r=1}^lm_{j,r}^2-\sum_{r=1}^la_{r}^2-\sum_{r=1}^l b_{r}^2 <
\sum_{j=1}^{l_0}d_j^2-p_1^2-p_2^2.
\]

For simplicity, define $s_2=2$ if $s_1=1$ and $s_2=1$ if $s_1=2$. Now pick $j_2\not= j_1$. 
Since $\mu_1(j)a_r=m_{j,r}=\mu_2(j)b_2$ we can write 
\begin{eqnarray*}
\sum_{j=1}^{l_0}\sum_{r=1}^lm_{j,r}^2-\sum_{r=1}^la_{r}^2-\sum_{r=1}^l b_{r}^2&=& 
\sum_{\substack{1\leq j \leq l_0 \\ j \not= j_1,j_2} }\sum_{r=1}^l m_{j,r}^2 \\
&+& \sum_{r=1}^l m_{j_1,r}^2\bigg(1-\frac{1}{\mu_{s_1}(j_1)^2 }\bigg) \\
&+& \sum_{r=1}^l m_{j_2,r}^2\bigg(1-\frac{1}{\mu_{s_2}(j_2)^2}\bigg).
\end{eqnarray*}

Now, taking into account that  $\sum_{r=1}^lm_{j,r}=d_j$ we deduce
\begin{equation}\label{Eqn:l0Big}
\sum_{\substack{1\leq j \leq l_0 \\ j \not= j_1,j_2} }\sum_{r=1}^l m_{j,r}^2 \leq \sum_{\substack{1\leq j \leq l_0 \\ j \not= j_1,j_2} }  d_{j}^2.
\end{equation}

Notice that if $l_0 \geq 3$, the inequality in (\ref{Eqn:l0Big}) is strict, since $l\geq 2$.

Similarly  we obtain
\begin{equation}\label{Eqn:MultiplicityBig}
\sum_{r=1}^l m_{j_1,r}^2\bigg(1-\frac{1}{\mu_{s_1}(j_1)^2}\bigg) \leq  d_{j_1}\bigg(1-\frac{1}{\mu_{s_1}(j_1)^2}\bigg). 
\end{equation}

Also notice that if $\mu_{s_1}(j_1) \geq 2$, then the inequality in (\ref{Eqn:MultiplicityBig}) is strict.

Since the term $1-/\mu_{s_2}(j_2)$ can be zero, we can only assert that
\[
\sum_{r=1}^l m_{j_2,r}^2\bigg(1-\frac{1}{\mu_{s_2}(j_2)^2}\bigg) \leq  d_{j_2}\bigg(1-\frac{1}{\mu_{s_2}(j_2)^2}\bigg) .
\]

We conclude that, if either $l_0 \geq 3$ or $\mu_{s_1}(j_1) \geq 2$ then
\begin{eqnarray*}
\sum_{j=1}^{l_0}\sum_{r=1}^lm_{j,r}^2-\sum_{r=1}^la_{r}^2-\sum_{r=1}^l b_{r}^2 &<& \sum_{\substack{1\leq j \leq l_0 \\ j \not= j_1,j_2} }  d_{j}^2 \\
&+&  d_{j_1}\bigg(1-\frac{1}{\mu_{s_1}(j_1)^2}\bigg) \\
&+&  d_{j_2}\bigg(1-\frac{1}{\mu_{s_2}(j_2)^2}\bigg)\\ &=&\sum_{j=1}^{l_0}d_j^2-p_1^2-p_2^2.
\end{eqnarray*}

\end{proof}

As a direct consequence we have the following corollary.

\begin{coro}\label{Coro:MatrxMatrxPrimitivityOneBigmu}
Assume 
\begin{enumerate}[(i)]
\item $A_1\simeq M_{n_1}$, $A_2\simeq M_{n_2}$, $n_1,n_2\geq 2$,
\item there are faithful traces on $A_1$ and $A_2$ whose restrictions agree on $D$,
\item $\dim(D)=l_0 \geq 2$,
\item either
\begin{enumerate}[(a)]
\item there is $s_1\in \{1,2\}$ and $j_1\in [l_0]$ such that $\mu_{s_1}(1,j_1) \geq 2$, 
\item $l_0 \geq 3$.
\end{enumerate}
\end{enumerate}
Then $A_1*_DA_2$ is primitive.
\end{coro}

Finally, the only case not covered is $M_2*_{\mathbb{C}^2}M_2$, which is not primitive by
Proposition  \ref{Prop:ExmapleM2M2}.

The shortest way to put all these results together is in the following theorem, whose proof uses
the reductions in section 2 (specially Corollary \ref{Coro:ReductionTrivialCenter} ) and all cases covered in this last section.

\begin{thm}
Let  $A_1,A_2$ and $D$ be finite dimensional C*-algebras.

Assume
\begin{enumerate}[(i)]
\item there are faithful traces on $A_1$ and $A_2$ whose restrictions to $D$ agree,
\item the ranks of $\mu_1$ and $\mu_2$ are both 1.
\end{enumerate}

Then $A_1*_DA_2$ is primitive if and only if its center is trivial.
\end{thm}

\subsection{Conjecture} 

From  our previous examples, it seems tempting to conjecture that the only
obstruction for a unital full free product of finite dimensional C*-algebras to be primitive is to have a non trivial center. Corollary
\ref{CoroPrimitivityBiggerThan2} says somehow that if there is a lot of gluing, we obtain a primitive C*-algebra, but
for subtle cases i.e. when the multiplicities are one, there seem to be a wealth of algebras that, so far, we can not manage
to classify. 

As supporting evidence for this conjecture, we present some test cases of the type $A_1*_{\mathbb{C}^2}A_2$, where $A_2$
is abelian and $A_1$ has minimum requirements so that $A_1*_{\mathbb{C}^2}A_2$ satisfies the LP condition.  We choose $A_2$ to be abelian because in this case
all the entries of $\mu_2$ are either 1 or 0, cases not covered so far. On the other hand, if we take $A_1$ abelian, the C*-algebra $A_1*_{\mathbb{C}^2}A_2$ will have a non-trivial center and then  not primitive. Thus the simplest case is to take $A_1$ to be
$M_2$. Hence, we are looking at algebras of the form 
\begin{equation}\label{Eqn:TestAlgebras}
 M_2 *_{\mathbb{C}^2}(\mathbb{C}^{N_2(1)}\oplus \mathbb{C}^{N_2(2)})
\end{equation}
$N_2(1), N_2(2) \geq 1$, where the inclusions are induced  by
\begin{eqnarray*}
\gamma_1(1,0)&=&  
\left[
\begin{array}{cc} 1 & 0 \\
0 & 0 
\end{array}
\right],\\
\gamma_1(0,1) &=&  
\left[
\begin{array}{cc} 0 & 0 \\
0 & 1
\end{array}
\right]
\end{eqnarray*}
and
\begin{eqnarray*}
\gamma_2(1,0)&=&(1,0)^{(N_2(1))},\\
\gamma_2(0,1)&=&(0,1)^{(N_2(2))}.
\end{eqnarray*}

The case $N_2(1)=N_2(2)=1$ is trivial so we start with $N_2(1)=1$ and $N_2(2) \geq 2$.

\begin{prop}
\[
M_2*_{\mathbb{C}^2}(\mathbb{C}\oplus \mathbb{C}^{N_2(2)})\simeq M_2(\mathbb{C}^{N_2(2)})
\]
where $N_2(2)\geq 2$. Since $\mathbb{C}^{N_2(2)}$ has a non-trivial center, so does  the
C*-algebra $M_2*_{\mathbb{C}^2}(\mathbb{C}\oplus \mathbb{C}^{N_2(2)})$.

\end{prop}
\begin{proof}
We are to show that $M_2(\mathbb{C}^{N_2(2)})$ has the universal property characterizing 
$M_2*_{\mathbb{C}^2}(\mathbb{C}\oplus \mathbb{C}^{N_2(2)})$. 

For simplicity write $n=N_2(2)$. Let $\{e_{i,j}\}$ be a matrix unit
for $M_2$, $\{e_s\}_{s=1}^{n}$ denote a set of minimal projections of
$\mathbb{C}^n$ and define $\iota_1:M_2\to M_2(\mathbb{C}^n)$ by
\[
\iota_2\bigg(\sum_{i,j}e_{i,j}\otimes x_{i,j} \bigg)=\sum_{i,j}e_{i,j}\otimes x_{i,j}1_{\mathbb{C}^n}
\]
and $\iota_2:\mathbb{C}\oplus \mathbb{C}^n\to M_2(\mathbb{C}^n) $ by
\[
\iota_2(x,y_1,\dots, y_n)=e_{1.1}\otimes x1_{\mathbb{C}^n} + e_{2,2}\otimes (y_1,\dots, y_n).
\]

Now we take $\varphi_1:M_2\to \mathbb{B}(H)$ and $\varphi_2:\mathbb{C}\oplus \mathbb{C}^n \to \mathbb{B}(H)$
such that $\varphi_1 \circ\gamma_1=\varphi_2 \circ\gamma_2$ and we will construct
$\varphi:M_2(\mathbb{C}^n) \to \mathbb{B}(H)$ such that $\varphi\circ \iota_i=\varphi_i$.

Let $E_{i,j}=\varphi_1(e_{i,j})$. We may assume $H=K\oplus K$, where $K=E_{i,i}(H)$. Thus
\[
\varphi_1\bigg( \sum_{i,j}e{i,j}\otimes x_{i,j}\bigg)=\sum_{i,j}E_{i,j}\otimes x_{i,j}id_K.
\] 

Since $\varphi_1\circ \gamma_1=\varphi_2\circ \gamma_2$, $E_{1,1}=\varphi_2(e_1)$ and 
$\sum_{s=2}^n \varphi_2(e_s)=E_{2,2}$. Define $\alpha:\mathbb{C}^n \to \mathbb{B}(K)$
the unital $*$-homomorphism induced by $\alpha(e_s)= E_{2,2}\varphi(e_s)E_{2,2}$.
Thus
\[
E_{1,1}\otimes xid_K+ E_{2,2}\otimes \alpha(y_1,\dots y_n)=\varphi_2(x,y_1,\dots y_n).
\]
Define
$\varphi:M_2(\mathbb{C}^n)\to \mathbb{B}(K)$ by
\[
\varphi\bigg(\sum_{i,j}e_{i,j}\otimes a_{i,j} \bigg)=\sum_{i,j}E_{i,j}\otimes \alpha(a_{i,j}).
\]

Then 
\begin{eqnarray*}
\varphi\bigg( \iota_1\bigg( \sum_{i,j}e_{i,j}\otimes x_{i,j} \bigg)\bigg)&=&
\varphi\bigg( \sum_{i,j}e_{i,j}\otimes x_{i,j}1_{\mathbb{C}^n}  \bigg)\\
&=&\sum_{i,j}E_{i,j}\otimes x_{i,j}id_K\\
&=&\varphi_1\bigg( \sum_{i,j}e_{i,j}\otimes x_{i,j}\bigg)
\end{eqnarray*}
and

\begin{eqnarray*}
\varphi(\iota_2(x,y_1,\dots, y_n))&=&\varphi(e_{1,1}\otimes x1_{\mathbb{C}^n} +e_{2,2}\otimes(y_1,\dots y_n) )\\
&=& E_{1,1}\otimes xid_K+ E_{2,2}\otimes\alpha(y_1,\dots y_n)\\
&=&\varphi_2(x,y_1,\dots ,y_n).  
\end{eqnarray*}

\end{proof}

The next natural step is to take  $N_2(1)=2$. It turns out that if $N_2(2)=2$ we don't get a primitive C*-algebra. 
Curiously, $N_2(2)\geq 3$ will produce primitive C*-algebras. This is proved in the next two propositions.

\begin{prop}
\[
M_2*_{\mathbb{C}^2}(\mathbb{C}^2\oplus \mathbb{C}^2)\simeq M_2(\mathbb{C}^2*_\mathbb{C}\mathbb{C}^2)
\]
Since $\mathbb{C}^2*_\mathbb{C}\mathbb{C}^2$ has a non-trivial center, so does $M_2*_{\mathbb{C}^2}(\mathbb{C}^2\oplus \mathbb{C}^2)$.
\end{prop}
\begin{proof}

We will show that $M_2(\mathbb{C}^2*_\mathbb{C}\mathbb{C}^2)$ has the universal property characterizing 
$M_2*_{\mathbb{C}^2}(\mathbb{C}^2\oplus \mathbb{C}^2)$. 

Let $B=\mathbb{C}^2*_\mathbb{C}\mathbb{C}^2$,
 $j_1,j_2$ will denote the inlclusions from $\mathbb{C}^2$ into $B$
and $\{e_{i,j}\}$ will denote a matrix unit for $M_2$. Define 
$\iota_1:M_2 \to M_2(B)$ by 
\[
\iota_1\bigg(\sum_{i,j}e_{i,j}\otimes x_{i,j} \bigg)=\sum_{i,j}e_{i,j}\otimes (x_{i,j}1_{B})
\]
and $\iota_2:\mathbb{C}^2\oplus \mathbb{C}^2 \to M_2(B)$ by
\[
\iota_2(x_1,x_2,x_3,x_4)=e_{1,1}\otimes j_1(x_1,x_2)+ e_{2,2}\otimes j_2(x_3,x_4).
\]

Now we take $*$-homomorphisms $\varphi_1:M_2 \to \mathbb{B}(H)$, 
$\varphi_2:\mathbb{C}^2\oplus \mathbb{C}^2\to \mathbb{B}(H)$ such that 
$\varphi_1\circ\gamma_1=\varphi_2\circ\gamma_2$ and we wil construct a $*$-homomorphism
$\varphi:M_2(B)\to \mathbb{B}(H)$ such that $\varphi\circ \iota_i=\varphi_i$.

Let $E_{i,j}=\varphi_1(e_{i,j})$. We may assume $H=K\oplus K$, where $K=E_{i,i}(H)$. Thus
\[
\varphi_1\bigg( \sum_{i,j}e_{i,j}\otimes x_{i,j}\bigg)=\sum_{i,j}E_{i,j}\otimes (x_{i,j}id_K).
\]
Since $\varphi_1\circ \gamma_1=\varphi_2\circ \gamma_2$, $\varphi_2(1,1,0,0)=E_{1,1}$
and $\varphi(0,0,1,1)=E_{2,2}$.  

Define $\alpha_i:\mathbb{C}^2 \to \mathbb{B}(K)$ by
$\alpha_1(x,y)=E_{1,1}\varphi_1(x,y,0,0)E_{1,1}$, $\alpha_2(x,y)=E_{2,2}\varphi_2(0,0,x,y)E_{2,2}$.
Then $\alpha_1$ and $\alpha_2$ are unital $*$-homomorphisms, so we may take $\alpha:=\alpha_1*\alpha_2$.
Define $\varphi:M_2(B)\to \mathbb{B}(K\oplus K)$ by
\[
\varphi\bigg(\sum_{i,j}e_{i,j}\otimes b_{i,j} \bigg) =\sum_{i,j}E_{i,j}\otimes b_{i,j}.
\]

Then
\begin{eqnarray*}
\varphi\bigg( \iota_1\bigg( \sum_{i,j} e_{i,j}\otimes x_{i,j} \bigg) \bigg)
&=&\varphi\bigg(\sum_{i,j}e_{i,j}\otimes (x_{i,j}1_B) \bigg)\\
&=& \sum_{i,j}E_{i,j}\otimes (x_{i,j}\alpha(1_B))=\sum_{i,j}E_{i,j}\otimes (x_{i,j}id_K)\\
&=&\varphi_1 \bigg( \sum_{i,j}e_{i,j}\otimes x_{i,j}\bigg)
\end{eqnarray*} 

and

\begin{eqnarray*}
\varphi(\iota_2(x_1,x_2,x_3,x_4))&=&\varphi(e_{1,1}\otimes j_1(x_1,x_2)+e_{2,2}\otimes j_2(x_3,x_4))\\
&=& E_{1,1}\otimes \alpha(j_1(x_1,x_2))+E_{2,2}\otimes \alpha(j_2(x_3,x_4))\\
&=& E_{1,1}\otimes \alpha_1(x_1,x_2)+E_{2,2}\otimes \alpha_2(x_3,x_4)\\
&=&\varphi_2(x_1,x_2,0,0)+\varphi_2(0,0,x_3,x_4)\\
&=&\varphi_2(x_1,x_2,x_3,x_4).
\end{eqnarray*}
\end{proof}

To finish, we will prove that the C*-algebras $M_2*_{\mathbb{C}^2}(\mathbb{C}^{N_2(1)} \oplus \mathbb{C}^{N_2(2)})$, with $N_2(1)\geq 3$ and $N_2(2) \geq 2$ are primitive. 

For simplicity, $A_1=M_2 $ and $A_2=\mathbb{C}^{N_2(1)}\oplus \mathbb{C}^{N_2(2)}$. From our reductions, we only need to
show that given any $\pi$, unital, injective finite dimensional $*$-representation, there is $\hat{\pi}$, another unital, injective
finite dimensional $*$-representation such that $\pi\oplus \hat{\pi} $ is DPI.

As before, given a unital, injective, finite dimensional $*$-representation
$\pi:A_1*_{\mathbb{C}^2}A_2 \to \mathbb{B}(H)$, we let
\begin{eqnarray*}
B_0&:=&\pi^{(0)}(\mathbb{C}^2)' \simeq M_{k_1}\oplus M_{k_2},\\
B_1&:=&\pi^{(1)}(A_1)' \simeq  M_p,\\
B_2&:=&\pi^{(2)}(A_2)'\simeq \oplus_{i=1}^{N_2(1)}M_{q_i(1)} \oplus \oplus_{i=1}^{N_2(2)}M_{q_i(2)}.
\end{eqnarray*}
The corresponding Bratelli diagram, for $B_1, B_2$ and  $B_0$ looks like:
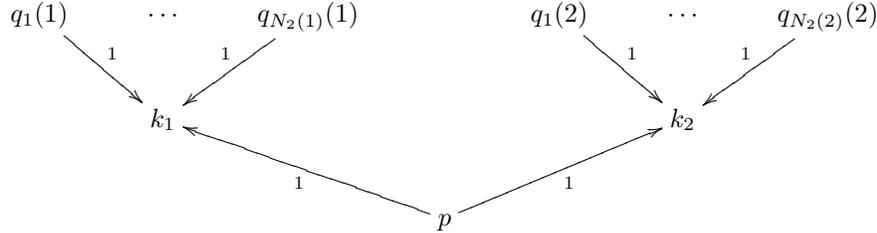
\[
\xymatrix{ 
q_1(1)\ar[dr]^1 &\cdots& q_{N_2(1)}(1)\ar[dl]_1 & &q_1(2)\ar[dr]^1&\cdots&q_{N_2(2)}(2) \ar[dl]_1 \\
                 & k_1 &                       & &              &  k_2 &                \\
                &        &          & p\ar[ull]^1\ar[urr]_1 &  &        &    & 
 }
\]

Notice that the weight of each arrow comes from the matrices $\mu(A_i,\mathbb{C}^2)$ and the
fact that $\mu(B_0,B_i)=\mu(A_i,\mathbb{C}^2)^t$. Then $k_1=k_2=p$.

\begin{rmk}

Consider  the representations 
\begin{eqnarray*}
\hat{\pi}_1&=&  id_{M_2}^{(\hat{p})} ,\\
\hat{\pi}_2&=&\oplus_{i=1}^{N_2(1)}id_{\mathbb{C}}^{(\hat{q}_i(1))}\oplus \oplus_{i=1}^{N_2(2)}id_{\mathbb{C}}^{(\hat{q}_i(2))}.
\end{eqnarray*}

The free product $\hat{\pi}:=\hat{\pi}_1*\hat{\pi}_2$ is well defined if and only if
\begin{eqnarray*}
\hat{p}&=&\sum_i \hat{q}_i(1)= \sum_i \hat{q}_i(2) .\\
\end{eqnarray*}

Indeed, this is equivalent to  $\hat{\pi}_1\circ \gamma_1=\hat{\pi}_2\circ \gamma_2$.
Notice that also this conditions implies both, $\hat{\pi}_1$ and $\hat{\pi}_2$, are representations
on $M_N$, where $N=\sum_i\hat{q}_i(1)+\sum_{i}\hat{q}_i(2)=2\hat{p}$
(the 2 in front of $\hat{p}$ comes from the block $M_2$ ).

Lastly, the Bratelli diagram for $(\pi \oplus \hat{\pi})^{(1)}(A_1)', (\pi \oplus \hat{\pi})^{(2)}(A_2)', (\pi \oplus \hat{\pi})^{(0)}(\mathbb{C}^2)',$ is

\[
\xymatrix{ 
\tilde{q}_1(1)\ar[dr]^1 &\cdots&\tilde{q}_{N_2(1)}(1) \ar[dl]_1 & 
&\tilde{q}_1(2)\ar[dr]^1&\cdots&\tilde{q}_{N_2(2)}(2) \ar[dl]_1 \\
                 & \tilde{k}_1 &                       & &           & \tilde{k}_2 &                \\
    &  &    & 
\tilde{p}\ar[ull]^1\ar[urr]_1 &   &   &    & 
 }
\]

where $\tilde{q}_i(j)=q_i(j)+\hat{q}_i(j),\tilde{k}_i=k_i + \hat{p}=p+\hat{p}$.

\end{rmk}

\begin{lema}\label{Lema:SpecialRep}
Given $\{q_i(1)\}_{i=1}^{N_2(1)}$ and $\{q_i(2)\}_{i=1}^{N_2(2)}$, such that
\[
\sum_i q_i(1)=\sum_i q_i(2)=p,
\]
there are positive integers $\{\hat{q}_i(1)\}_{i=1}^{N_2(1)}, \{\hat{q}_i(2)\}_{i=1}^{N_2(2)}$
and $\hat{p}$ such that 
\begin{enumerate}[(i)]
\item $q_i(1)+\hat{q}_i(1)$ is independent of $i$,
\item $q_i(2)+\hat{q}_i(2)$ is independent of $i$,
\item and

\begin{eqnarray*}
\hat{p}&=&\sum_i \hat{q}_i(1)=\sum_i \hat{q}_i(2).
\end{eqnarray*}

Note: in general, $\hat{q}_i(1)+q_i(1)\not= \hat{q}_i(2)+q_i(2)$.
\end{enumerate}

\end{lema}
\begin{proof}
For  $j=1,2$, let $q(j)=\lcm_i(q_i(j))$ and let
\begin{eqnarray*}
Q(1)&=&Qq(2)N_2(2),\\
Q(2)&=&Qq(1)N_2(1),
\end{eqnarray*}
where $Q$ is a positive integer that will be specified later on.

Define
\begin{eqnarray*}
\hat{q}_i(j)&:=&Q(j)q(j)-q_i(j),\\
\hat{p}&:=&Qq(1)q(2)N_2(1)N_2(2)-p.
\end{eqnarray*}

If we take $Q$ large enough, $\hat{q}_i(j)$ and $\hat{p}$ are positive. It is plain that
$\hat{q}_i(j)+q_i(j)$ is independent of $i$. We are left to show

\begin{equation}\label{Eqn:ConditionFromAmalgamation}
\sum_{i=1}^{N_2(j)}\hat{q}_i(j)=\hat{p},
\end{equation}
for $j=1,2$.

For $j=1$, the left hand side of (\ref{Eqn:ConditionFromAmalgamation}) equals
\begin{eqnarray*}
Q(1)q(1)N_2(1)-\sum_{i=1}^{N_2(1)}q_i(1)&=& Q(1)q(1)N_2(1)-p
\\
&=&Qq(1)q(2)N_2(1)N_2(2)-p=\hat{p}.
\end{eqnarray*}

Similarly for $j=2$.

\end{proof}

Notice that if $\hat{\pi}$ is the injective unital  $*$-representation of $A_1*_{\mathbb{C}^2}A_2$
induced by the numbers $\hat{q}_i(j),\hat{p}$ in the previous lemma, the Bratelli diagram for
$(\pi \oplus \hat{\pi})^{(1)}(A_1)', (\pi \oplus \hat{\pi})^{(2)}(A_2)' $ and $(\pi \oplus \hat{\pi})^{(0)}(\mathbb{C}^2)'$ is

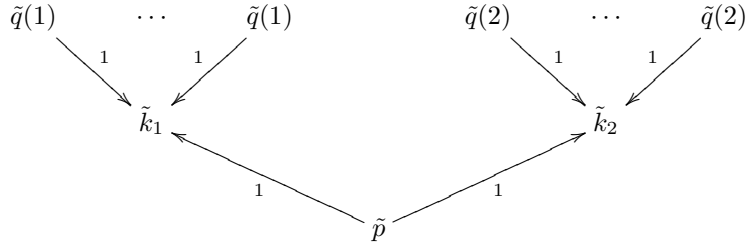
\[
\xymatrix{ 
\tilde{q}(1)\ar[dr]^1 &\cdots&\tilde{q}(1) \ar[dl]_1 & 
&\tilde{q}(2)\ar[dr]^1&\cdots&\tilde{q}(2) \ar[dl]_1 \\
                 & \tilde{k}_1 &                       & &           & \tilde{k}_2 &                \\
 &    &    & 
\tilde{p} \ar[ull]^1\ar[urr]_1 &   &   &   & 
 }
\]
where $\tilde{q}(j)=q_i(j)+\hat{q}_i(j)$ and $\tilde{k}_1=\tilde{k}_2=\tilde{p}=p+\hat{p}$.

To finish this section, given $\pi$ we construct $\hat{\pi}$ using Lemma \ref{Lema:SpecialRep} and  show that if $N_2(1) \geq 3$ and $N_2(2) \geq 2$, then the $*$-representation  $\pi \oplus \hat{\pi}$ is DPI.

\begin{prop}
The C*-algebras,
\[
M_2*_{\mathbb{C}^2}(\mathbb{C}^{n_2(1)}\oplus \mathbb{C}^{n_2(2)})
\]
where  $N_2(1)\geq 3$ and $N_2(2) \geq 2$, are primitive.

\end{prop}
\begin{proof}

As we mentioned before, it suffices to show that  $\pi \oplus \hat{\pi}$ is DPI and to show it we will use Proposition \ref{Prop:DensityHardVersion}. As usual, $B_0=(\pi \oplus \hat{\pi})(\mathbb{C}^2)', B_1=(\pi \oplus \hat{\pi})(A_1)'$
and $B_2=(\pi \oplus \hat{\pi})(A_2)'$. 

For simplicity, let $p=\tilde{p},q_1=\tilde{q}(1)$ and $q_2=\tilde{q}(2)$. The corresponding Bratelli diagram looks (the numbers within parenthesis indicate how many direct summands
we have)

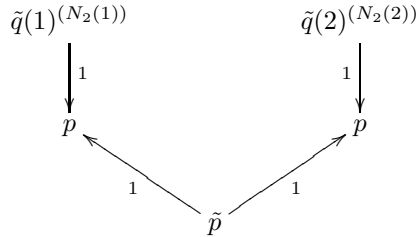
\[
\xymatrix{ 
\tilde{q}(1)^{( N_2(1))}\ar[d]^1 & &\tilde{q}(2)^{(N_2(2))}\ar[d]_1 \\ 
        p  &                    & p &                \\
 &              \tilde{p}\ar[ur]_1\ar[ul]^1 & 
 }
\]

Thus, we need to show
\begin{equation}\label{Eqn:MatrixMatrixEasyDos}
\dim(\mathbb{U}(B_1))+\dim(\mathbb{U}(B_2))\leq \mathbb{U}(B_0)
\end{equation}
and
\begin{equation}\label{Eqn:MatrixMatrixHardDos}
\left. 
\begin{array}{ll}
\hspace{0.4cm} \dim(\mathbb{U}(B_1)) - \dim(\mathbb{U}(B_1\cap C'))  &   \\
  +\dim(\mathbb{U}(B_2))-  \dim(\mathbb{U}(\Ad u(B_2)\cap C')) &   <  \dim(\mathbb{U}(B_0))\\
  +\dim(\mathbb{U}(B_0\cap C'))  &  \\
\end{array}
\right\}
\end{equation}
where $C$ is an abelian unital C*-subalgebra of $B_1$ with $\dim(C)=l \geq 2$, $C'$ is the commutant relative to $\mathbb{B}(H\oplus \hat{H})$ and $u\in \mathbb{U}(B_0)$ is such that $C\subseteq Ad u(B_2)$.

Next we prove  (\ref{Eqn:MatrixMatrixEasyDos}). We have that $\dim(\mathbb{U}(B_1))=p^2, \dim(\mathbb{U}(B_1))=N_2(1)q_1^2+N_2(2)q_2^2$ and $\dim(\mathbb{U}B_0))=2p^2$. But taking into account that $N_2(1)q_1=N_2(2)q_2=p$, we simplify and get  $\dim(\mathbb{U}(B_2))=\frac{p^2}{N_2(1)}+\frac{p^2}{N_2(2)}$. Hence
(\ref{Eqn:MatrixMatrixEasyDos}) follows from the fact that $N_2(1)\geq 3$ and $N_2(2)\geq 2$.

Now, let us denote 

\begin{eqnarray*}
\mu(B_1,C)&=&[a_{1,r}]_{ 1 \leq r \leq l},\\
\mu(B_2,u^*Cu)&=&[b_{i,r}]_{1\leq i \leq l_2, 1\leq r \leq l},\\
\mu(B_0,C)&=&[m_{j,r}]_{1\leq j \leq 2, 1\leq r \leq l}.\\
\end{eqnarray*}

With this notation (\ref{Eqn:MatrixMatrixHardDos}) becomes
\[
\frac{1}{N_2(1)}p^2+\frac{1}{N_2(2)}p^2+p^2 +\sum_{ \substack{1 \leq j \leq  2  \\  1 \leq r  \leq l } }m_{j,r}^2-
\sum_{\substack{1\leq r \leq l } }a_{1,r}^2-
\sum_{\substack{1\leq i \leq l_2 \\ 1\leq r \leq l } }b_{i,r}^2< 2p^2,
\]
or, equivalently,
\[
\sum_{j,r }m_{j,r}^2-
\sum_{r }a_{1,r}^2-
\sum_{i,r }b_{i,r}^2< \bigg(1-\frac{1}{N_2(1)}-\frac{1}{N_2(2)} \bigg)p^2.
\]

With no loss of generality we may assume $\sum_r m_{2,r}^2 \leq \sum_r m_{1,r}^2$.

Since $\mu_1(1,j)=1$ for all  $j$, $a_{1,r}=m_{j,r}$ for all $j$ and $r$. Hence we simplify the previous inequality to get
\begin{equation}\label{Eqn:SimplificationHard}
\sum_{r}m_{2,r}^2-\sum_{i,r}b_{i,r}^2 < \bigg(1-\frac{1}{N_2(1)}-\frac{1}{N_2(2)} \bigg)p^2.
\end{equation}

Now, notice that $\mu_2(1,j)=0$ for $j\geq N_2(1)+1$ and 1 otherwise. Similarly,  $\mu_2(2,j)=0$ for $j\leq N_2(1)$ and
1 otherwise (observe that $l_2=N_2(1)+N_2(2)$). Thus, we get
\[
m_{1,r}=\sum_{i=1}^{N_2(1)}b_{i,r}, m_{2,r}=\sum_{i=N_2(1)+1}^{l_2}b_{i,r}.
\]
In consequence
\begin{eqnarray*}
\sum_{i=1}^{N_2(1)}b_{i,r}^2 \geq \frac{m_{1,r}^2}{N_2(1)},\\
\sum_{i=N_2(1)+1}^{l_2}b_{i,r}^2 \geq \frac{m_{2,r}^2}{N_2(2)},
\end{eqnarray*}

which in turn bring the estimate
\begin{eqnarray*}
\sum_r m_{2,r}^2-\sum_{i,r}b_{i,r}^2 &\leq&   \sum_{r}m_{2,r}^2-\frac{1}{N_2(1)}\bigg(\sum_{r=1}^l m_{2,r}^2\bigg)-\frac{1}{N_2(1)}\bigg(\sum_{r=1}^l m_{1,r}^2  \bigg)
\\
&\leq& \bigg(1-\frac{1}{N_2(1)}-\frac{1}{N_2(2)} \bigg) \bigg(\sum_{r=1}^l m_{2,r}^2\bigg).
\end{eqnarray*}

Lastly, since $\sum_r m_{2,r}=p$ and $l\geq 2$, (\ref{Eqn:SimplificationHard}) holds.

\end{proof}

\section*{Acknowledgments}
Research supported by DGAPA-UNAM.

\end{document}